\documentclass[11pt]{article}

\usepackage{geometry}
\usepackage{amsmath,amsthm,amsfonts,amssymb,epsfig,graphicx,color, float}

\usepackage[colorlinks, linkcolor=blue,citecolor=blue]{hyperref}
\usepackage{pdfpages}
\usepackage{graphicx}

\usepackage{natbib}
\usepackage{dsfont}

\numberwithin{table}{section}
\numberwithin{figure}{section}
\numberwithin{equation}{section}

\definecolor{darkblue}{rgb}{.2, 0.2,.8}
\definecolor{darkgreen}{rgb}{0,0.5,0.3}
\definecolor{darkred}{rgb}{.8, .1,.1}

\newcommand{\bfalp}{\vect{\alpha}}

\newtheorem{lemma}{Lemma}[section]

\newtheorem{proposition}[lemma]{Proposition}

\newtheorem{corollary}[lemma]{Corollary}

\newtheorem{remark}{Remark}[section]

\usepackage{bm}
\newcommand{\vect}[1]{\pmb{#1}}
\newcommand{\mat}[1]{\boldsymbol{\bm #1}}

\usepackage{marginnote}

\allowdisplaybreaks

\begin{document}
\title{
\bf\
{Modeling discrete common-shock risks through matrix
distributions}
}

\author{Martin Bladt, Eric C. K. Cheung, Oscar Peralta and Jae-Kyung Woo\\
\\
}
\date{}
%\today
\bibliographystyle{apalike}

%\title{Modelling and estimating discrete-valued common shocks via matrix distributions}
\maketitle

\begin{abstract}
We introduce a novel class of bivariate common-shock discrete phase-type (CDPH) distributions to describe dependencies in loss modeling, with an emphasis on those induced by common shocks. By constructing two jointly evolving terminating Markov chains that share a common evolution up to a random time corresponding to the common shock component, and then proceed independently, we capture the essential features of risk events influenced by shared and individual-specific factors. We derive explicit expressions for the joint distribution of the termination times and prove various class and distributional properties, facilitating tractable analysis of the risks. Extending this framework, we model random sums where aggregate claims are sums of continuous phase-type random variables with counts determined by these termination times, and show that their joint distribution belongs to the multivariate phase-type or matrix-exponential class. We develop estimation procedures for the CDPH distributions using the expectation-maximization algorithm and demonstrate the applicability of our models through simulation studies and an application to bivariate insurance claim frequency data.
\end{abstract}

%\keywords{...}
%\subjclass{Primary; Secondary }


\section{Introduction}

Ever since their conception in \cite{jensen1954distribution}, phase-type distributions, which represent the time until absorption in a finite state-space Markov chain, have been extensively used due to their flexibility in approximating a wide range of distributions and their analytical tractability (see \cite{bladt2017matrix} for a review). Such a class of distributions has also gained popularity in actuarial science in the last two decades. In insurance ruin problems, phase-type distributions can be used to model the claim amounts (e.g. \cite{Drekic2004ruinPHclaims, Frostig2012ruinnumberPHclaim}) and/or the inter-arrival times (an assumption implicitly embedded in the class of risk processes with Markovian claim arrivals; see e.g. \cite{Cheung2010MAPruin}). Interested readers are referred to \cite{Asmussen2020RunProb} and \cite{Badescu2009MAPreview} for comprehensive reviews.
%For example, ruin-related quantities such as the deficit, the ruin time and the number of claims until ruin were analyzed by \cite{Drekic2004ruinPHclaims,Frostig2012ruinnumberPHclaim} analyzed in an insurance risk process with phase-type claim amounts. The class of risk processes with Markovian claim arrivals, which contains renewal models with phase-type inter-arrival times as special cases, have also been widely studied (see e.g. \cite{Cheung2010MAPruin}).
Outside ruin theory, \cite{Zadeh2016PHcredibility} developed Bayesian and {B}\"uhlmann credibility theories under phase-type claims whereas \cite{wang2018renewalsum} derived the distribution of the discounted aggregate claims under phase-type renewal process. Generalizations of phase-type distributions have also been proved to be successful for fitting insurance loss data. In particular, \cite{Ahn2012logPH} utilized log phase-type distribution to fit the well-known Danish fire data, and \cite{bladt2023PHMOEseverity} proposed a phase-type mixture-of-experts regression model to fit the severities from the French Motor Third Party Liability dataset. Apart from modeling insurance losses and interclaim times, phase-type mortality law was studied by \cite{Lin2007PHmortality}, and a review of the use of phase-type distributions in health care systems can be found in \cite{Fackrell2009PHhealthreview}. While all the above works are concerned with continuous phase-type distributions, discrete phase-type (DPH) distributions have also been applied to model claim frequencies. For example, Panjer-type recursions for the calculations of compound distributions under DPH claim counts were developed by \cite{E2006PHrecursion,wu2010PHrecursion,Ren2010PHrecursion}. %In the case of phase-type claim amounts, \cite{Hipp2006PanjerRecursionPH} found that the Panjer recursions can be considerably simplified.

In modern risk theory and insurance mathematics, there is increasing need for researchers and practitioners to accurately model the dependencies among multiple risk factors or risk events, as this is crucial for effective risk management. Such dependence modeling is particularly important in scenarios where different risks are influenced by shared factors or catastrophic events—commonly referred to as \emph{common shocks}, leading to realization of random variables that are strongly dependent. For example, an accident or a natural disaster can cause claims in both personal injury and property damage, resulting in correlated claim counts and dependent losses in different insurance contracts or business lines. Capturing dependencies is essential for a realistic assessment of joint behavior of risks. While the use of copulas is popular (see e.g. \cite{Frees1998copula} for a review), specific models for multivariate claim counts (e.g. \cite{Shi2014multiBNcounts, Pechon2018MultiClaimCount, Bolance2019multiGLMcounts, Fung2019MOEcorrelatedcounts}) and multivariate losses (e.g. \cite{lee2012MixEr, willmot2015MixEr}) have been developed as well. The presence of dependency has also led to related research topics in actuarial science, such as multivariate risk measures (see e.g. \cite{Cossette2016MultiTVaR, landsman2016multivariate}) and multivariate ruin problems (see e.g. \cite{Badila2015multiruin, Albrecher2022multiruinLaguerre}).

 Dependencies also commonly arise in the applied probability literature, and several approaches to constructing continuous multivariate phase-type distributions have been proposed in the field. These models often introduce dependencies by coupling the underlying Markov processes in specific ways, namely, by employing the exit times from different sets as in \cite{assaf1984multivariate}, or by using accumulating reward systems as in \cite{kulkarni1989new}. In \cite{bladt2010multivariate}, a class of multivariate matrix-exponential (MME) distributions, characterized by rational multivariate Laplace transforms, was introduced. However, the multivariate matrix distributions in \cite{kulkarni1989new} and \cite{bladt2010multivariate} do not necessarily admit closed-form probability density function in general, possibly limiting their applications. With an emphasis on explicit results and ease of applications in an actuarial context, recent efforts have also been made in \cite{cheung2022multivariate, bladt2023tractable, albrecher2023joint, bladt2023robust} to introduce new classes of distributions which are conditionally independent mixtures of matrix distributions. Specifically, in \cite{bladt2023tractable}, the author proposed to employ a system of Markov jump processes that agree on the initial state and evolve independently from that point onward. 

Despite the increasing interest in continuous multivariate matrix distributions mentioned above, extensions of the DPH class to a multivariate setting, particularly in the context of matrix distributions, are relatively scarce in the literature. One of the first attempts is the work by \cite{he2016analysis}, who constructed a multivariate DPH distribution via the numbers of different types of batches that have arrived before absorption of a discrete-time Markov chain, with an accompanying estimation scheme discussed in \cite{he2016parameter}. A discrete analog of \cite{kulkarni1989new}'s continuous multivariate phase-type distributions was  considered in \cite{navarro2019order}; however, such a class, like its continuous counterpart, does not admit closed-form expressions for the joint probability mass function (pmf). A more recent development was presented by \cite{bladt2023robust} which is a discrete version of \cite{bladt2023tractable}'s work, where the authors investigated the termination time of discrete-time Markov chains that begin in the same initial state but otherwise behave independently. However, many of these models impose dependencies that do not necessarily explicitly capture the common-shock scenarios that arise.

In this paper, we introduce a novel class of bivariate \emph{common-shock discrete phase-type} (CDPH) distributions. Our approach extends the DPH framework to a bivariate setting by constructing two temporarily jointly evolving Markov chains that share a common evolution up to a random time—the occurrence of the common shock (or the number of common shocks)—after which they evolve independently. This construction effectively models the dependency induced by common shocks while retaining the tractability of DPH distributions. More precisely, we consider two discrete-time Markovian processes, $M_1$ and $M_2$, which evolve identically until a decoupling time $\tau_{1,2}$, after which they proceed independently and terminate at possibly different times $\tau_1$ and $\tau_2$. We derive explicit expressions for the joint distribution of $(\tau_1, \tau_2)$, which characterize the CDPH distributions. This framework allows us to model scenarios where two types of risks are affected by shared events before diverging due to independent factors. We can show that, in the bivariate case, our proposed class of CDPH distributions is larger than the class of multivariate DPH distributions in \cite{bladt2023robust} and is a subclass of that considered by \cite{navarro2019order}. Building on this foundation, we explore applications of the CDPH distributions in risk modeling involving random sums with common shocks. In particular, we consider aggregate risks represented as sums of random variables whose counts are determined by $\tau_1$ and $\tau_2$. This setup is particularly relevant in insurance, where aggregate claims may depend on both shared catastrophic events and independent claim occurrences. We characterize the joint distributions of these random sums, demonstrating the practicality and analytical convenience of our models.

Our contributions can be summarized as follows: We introduce the CDPH distributions, a new class of bivariate DPH distributions designed to effectively model dependencies induced by common shocks. To facilitate their applications in loss modeling, we provide explicit formulas for the joint pmf and probability generating function (pgf) of the CDPH class and prove various closure properties. Characterizations of the associated random sums are also provided. Furthermore, we develop estimation procedures for the CDPH distributions using the expectation-maximization (EM) algorithm, enabling their statistical implementation. Finally, we illustrate the applicability and effectiveness of our models through simulation studies, highlighting their potential in practical risk assessment scenarios. %Although we formulate our research problem via common shocks, our proposed class of CDPH distributions is seen to be an effective model to fit or other dependency structures.
In essence, our work extends the versatility of phase-type distributions in modeling complex dependencies in arising in actuarial science, offering both theoretical insights and practical tools for actuaries and risk managers.

The remainder of the paper is organized as follows. In Section \ref{sec:common-shocks}, we define the CDPH distributions and derive their fundamental properties. Section \ref{sec:random-sums} applies these distributions to analyze random sums. In Section \ref{sec:estimation}, we address the estimation of the CDPH distributions from data, proposing an EM algorithm tailored to our models. Section \ref{sec:simulation} presents simulation studies that demonstrate the implementation and performance of our estimation procedures, as well as an application to a dataset comprising bivariate insurance claim frequencies. Finally, Section \ref{sec:conclusion} concludes the paper and suggests directions for future research.

\section{A new class of CDPH distributions}\label{sec:common-shocks}

This section provides some of the essential properties of CDPH distributions, starting from basic principles and then treating more intricate class properties. Throughout, we link our class to related constructions in the literature.

\subsection{Background on DPH}\label{subsec:DPHDef}

Consider a Markov chain $M=\{M(n)\}_{n\in\mathbb{N}}$ evolving in a finite state-space $\mathcal{E}=\{1, 2, \dots, d\}$ according to some initial probability vector $\bm{\alpha} = (\alpha_1, \alpha_2, \dots, \alpha_d)$ and subprobability matrix $\bm{P}=\{p_{ij}\}_{i,j \in \mathcal{E}}$, where $\mathbb{N}$ is the set of non-negative integers. Here, $M$ is assumed to be (almost surely) terminating, meaning that while in state $i$, it gets absorbed in the next step with probability $1 - \sum_{j=1}^{d} p_{ij}$. We note that if $M$ has no superfluous states, then the property of $M$ terminating is equivalent to the eigenvalues of $\bm{P}$ being strictly contained in the complex unit circle (see e.g. Theorem 1.2.63 in \cite{bladt2017matrix}).

An alternative setup is to consider the augmented state-space $\mathcal{E} \cup \Delta$, where $\Delta$ is an absorbing cemetery state used to indicate the termination of $M$. In this setup, $M$ evolves according to the probability vector $(\alpha_1, \alpha_2, \dots, \alpha_d, 0)$ and a block-partitioned transition matrix
\[
\begin{pmatrix}
\bm{P} & \bm{1} - \bm{P}\bm{1} \\
\bm{0} & 1
\end{pmatrix},
\]
where $\bm{1}$ denotes a column vector of ones of appropriate dimension ($d$ in this case), $\bm{P}\bm{1}$ corresponds to the sum of probabilities from each state to all others in $\mathcal{E}$, and $\bm{0}$ denotes a zero matrix of appropriate dimension (a zero row vector of dimension $d$ in this case). The final `1' is a scalar representing the probability of remaining in the state $\Delta$ once entered. Both setups—the terminating and the absorbing ones—are interchangeable for our purposes, and we switch back and forth depending on the situation.

Define $\tau$ to be the step in which $M$ gets terminated, that is, let
\[
\tau:=\inf\{n\in\mathbb{N}_+: M(n)\notin\mathcal{E}\},
\]
 where $\mathbb{N}_+$ is the set of positive integers. Then $\tau$ follows what is known as a DPH distribution, which we denote by $\operatorname{DPH}(\bm{\alpha},\bm{P})$. The class of DPH distributions is particularly tractable (see e.g. Chapter 1.2.6 in \cite{bladt2017matrix}). For instance, its pmf over $\mathbb{N}_+$ is explicitly given by
\begin{align*}
f_\tau(\ell) & := \mathbb{P}(\tau=\ell) = \bm{\alpha}\bm{P}^{\ell - 1} \left(\bm{1}-\bm{P}\bm{1}\right).%= \mathbb{P}(M({\ell-1})\neq\Delta, M(\ell)=\Delta) \\
%& = \sum_{i=1}^d \sum_{j=1}^d \mathbb{P}(M({0})=i, M({\ell-1})=j, M(\ell)=\Delta) \\
%& = \sum_{i=1}^d \sum_{j=1}^d \mathbb{P}(M({0})=i) \,\mathbb{P}(M({\ell-1})=j\mid M(0)=i) \,\mathbb{P}(M(\ell)=\Delta\mid M({\ell-1})=j) \\
%& = \sum_{i=1}^d \sum_{j=1}^d \alpha_i \left(\bm{P}^{\ell-1}\right)_{ij} \left(\bm{1}-\bm{P}\bm{1}\right)_{j}
\end{align*}
Likewise, the pgf takes the tractable form
\begin{align*}
\mathbb{E}(x^{\tau}) = x \bm{\alpha}(\bm{I}-x\bm{P})^{-1} \left(\bm{1}-\bm{P}\bm{1}\right),%& = \sum_{\ell=1}^\infty x^{\ell} f_\tau(\ell) \\
%& = \sum_{\ell=1}^\infty x^{\ell}\bm{\alpha}\bm{P}^{\ell - 1} \left(\bm{1}-\bm{P}\bm{1}\right) \\
%& = x\sum_{\ell=0}^\infty \bm{\alpha}(x\bm{P})^{\ell} \left(\bm{1}-\bm{P}\bm{1}\right)
\end{align*}
where $\bm{I}$ denotes the identity matrix. Note that the above formula is valid for all $x\in\mathbb{R}$ on the real line such that the eigenvalues of $x\bm{P}$ lie within the complex unit circle (the set that contains the interval $[-1,1]$).

\subsection{Definition, joint probability and joint moments of CDPH}\label{subsec:CDPHDef}

As mentioned in the introduction, literature on the construction of multivariate DPH distributions is rather limited. Our paper is inspired by the approach in \cite{bladt2023robust}, adding the flexibility that the Markov chains evolve synchronously for some period of time (rather than just starting in the same initial state) before proceeding independently. The exact details behind our novel construction are as follows. Consider two jointly evolving terminating processes, \( M_1 = \{ M_1(n) \}_{n \in \mathbb{N}} \) and \( M_2 = \{ M_2(n) \}_{n \in \mathbb{N}} \), each taking values in the state space \( \mathcal{E} \cup \mathcal{S} \), where \( \mathcal{E} \cap \mathcal{S} = \emptyset \) and \( |\mathcal{E}| \wedge |\mathcal{S}| \geq 1 \). We construct the joint process \( (M_1, M_2) \) such that each marginal process is a Markov chain, with a specific joint dependence structure: we let \( M_1(n) = M_2(n) \) for all \( n \) up to the first time they exit the set \( \mathcal{E} \). After this point, they evolve independently within \( \mathcal{S} \) and eventually terminate, possibly at different times.

More explicitly, we let \( M_1 \) be a terminating Markov chain that evolves according to the block-partitioned initial probability vector \( (\bm{\alpha}, \bm{0}) \) and subprobability matrix
\[
\bm{P}_1:=
\begin{pmatrix}
\bm{P} & \bm{U} \\
\bm{0} & \bm{Q}_1
\end{pmatrix},
\]
where \( \bm{P} \) is a subprobability matrix over \( \mathcal{E} \) (like before), \( \bm{U} \) denotes the transition probabilities from \( \mathcal{E} \) to \( \mathcal{S} \), \( \bm{Q}_1 \) is a subprobability matrix over \( \mathcal{S} \). It is further assumed that each row of the submatrix $(\bm{P}\,\,\,\bm{U})$ sums to one, so that the Markov chain $M_1$ (and also $M_2$ below) must enter $\mathcal{S}$ before termination. For \( M_2 \), we define \( M_2(n) = M_1(n) \) for all \( n \leq \tau_{1,2} \), where
\[
\tau_{1,2} := \inf\{ n \in \mathbb{N}_+ : M_1(n) \notin \mathcal{E} \}.
\]
Obviously, $\tau_{1,2}$ follows a $\operatorname{DPH}(\bm{\alpha},\bm{P})$ distribution. After time \( \tau_{1,2} \), we let \( M_2 \) evolve independently as a terminating Markov chain within \( \mathcal{S} \) according to the subprobability matrix \( \bm{Q}_2 \). This construction yields \( M_2 \) as a Markov chain in \( \mathcal{E} \cup \mathcal{S} \) that evolves according to the block-partitioned initial probability vector \( (\bm{\alpha}, \bm{0}) \) and subprobability matrix
\[
\bm{P}_2:=
\begin{pmatrix}
\bm{P} & \bm{U} \\
\bm{0} & \bm{Q}_2
\end{pmatrix}.
\]
We now define the bivariate random variable \( (\tau_1, \tau_2) \) via the termination times of \( M_1 \) and \( M_2 \) so that, for \( k \in \{1, 2\} \),
\begin{align*}
\tau_k := \inf\{ n \in \mathbb{N}_+ : M_k(n) \notin \mathcal{E} \cup \mathcal{S} \}.
%\tau_2 &:= \inf\{ n \in \mathbb{N}_+ : M_2(n) \notin \mathcal{E} \cup \mathcal{S} \}.
\end{align*}
For each \( k \in \{1, 2\} \), the marginal distribution of $\tau_k$ is clearly $\operatorname{DPH}((\bm{\alpha}, \bm{0}),\bm{P}_k)$. Moreover, \( \tau_1 \) and \( \tau_2 \) are greater than \( \tau_{1,2} \). Specifically, by writing $\tau_1=\tau_{1,2} + (\tau_1-\tau_{1,2})$ and $\tau_2=\tau_{1,2} + (\tau_2-\tau_{1,2})$, one observes that $\tau_1$ and $\tau_2$ are dependent for two reasons:
\begin{itemize}

\item Both $\tau_1$ and $\tau_2$ share the term $\tau_{1,2}$, the time until which the Markov chains $M_1$ and $M_2$ evolve identically. The random variable $\tau_{1,2}$ can be regarded as the common shock component.

\item Although the Markov chains $M_1$ and $M_2$ evolve independently after time $\tau_{1,2}$, the remaining times until they get absorbed (namely $\tau_1-\tau_{1,2}$ and $\tau_2-\tau_{1,2}$) are only conditionally independent given the state $M_1(\tau_{1,2})$ of the Markov chain $M_1$ at the time $\tau_{1,2}$. Unconditionally, the variables $\tau_1-\tau_{1,2}$ and $\tau_2-\tau_{1,2}$ are dependent.

\end{itemize}
The joint distribution of \( (\tau_1, \tau_2) \) is provided in the following proposition.

\begin{proposition}\label{th:mass-taua-taub}
For \( n_1, n_2 \geq 2 \), \( m, z_1, z_2 \geq 1 \), and \( i \in \mathcal{E} \), define  the unconditional and conditional joint pmf's of \( (\tau_1, \tau_2) \) by
\begin{align*}
f_{\tau_1,\tau_2}(n_1, n_2) &:= \mathbb{P}(\tau_1 = n_1, \tau_2 = n_2), \\
f_{\tau_1,\tau_2|M_1,M_2}(n_1, n_2|i) &:= \mathbb{P}(\tau_1 = n_1, \tau_2 = n_2 \mid M_1(0) = M_2(0) = i),
\end{align*}
and define the auxiliary functions
\begin{align*}
\psi(m, z_1, z_2) &:= \mathbb{P}(\tau_{1,2} = m, \tau_1 - \tau_{1,2} = z_1, \tau_2 - \tau_{1,2} = z_2), \\
\psi_i(m, z_1, z_2) &:= \mathbb{P}(\tau_{1,2} = m, \tau_1 - \tau_{1,2} = z_1, \tau_2 - \tau_{1,2} = z_2 \mid M_1(0) = M_2(0) = i).
\end{align*}
Then, we have
\begin{align}
f_{\tau_1,\tau_2}(n_1, n_2) &= \sum_{m=1}^{\min\{n_1, n_2\} - 1} \psi(m, n_1 - m, n_2 - m), \label{eq:phi1} \\
f_{\tau_1,\tau_2|M_1,M_2}(n_1, n_2|i) &= \sum_{m=1}^{\min\{n_1, n_2\} - 1} \psi_i(m, n_1 - m, n_2 - m), \label{eq:phik1}
\end{align}
and
\begin{align}
\psi(m, z_1, z_2) &= \bm{\alpha} \bm{P}^{m - 1} \bm{U} \left( (\bm{Q}_1^{z_1 - 1} \bm{q}_1) \odot (\bm{Q}_2^{z_2 - 1} \bm{q}_2) \right), \label{eq:psi1} \\
\psi_i(m, z_1, z_2) &= \bm{e}_i^\intercal \bm{P}^{m - 1} \bm{U} \left( (\bm{Q}_1^{z_1 - 1} \bm{q}_1) \odot (\bm{Q}_2^{z_2 - 1} \bm{q}_2) \right), \label{eq:psik1}
\end{align}
where \( \bm{q}_1 = \bm{1} - \bm{Q}_1 \bm{1} \) and \( \bm{q}_2 = \bm{1} - \bm{Q}_2 \bm{1} \) denote the one-step probabilities of termination from the states in \( \mathcal{S} \) for \( M_1 \) and \( M_2 \), respectively; \( \bm{e}_i \) denotes the \( i \)-th canonical column vector, and \( \odot \) denotes the Hadamard (entrywise) product of two vectors.
\end{proposition}

\begin{proof}
Equations \eqref{eq:phi1} and \eqref{eq:phik1} follow by straightforward application of the law of total probability. Next,
\begin{align*}
&\mathbb{P}(\tau_{1,2} = m, \tau_1 - \tau_{1,2} = z_1, \tau_2 - \tau_{1,2} = z_2) \\
&\quad = \sum_{j \in \mathcal{S}} \mathbb{P}(\tau_{1,2} = m, M_1(m) = M_2(m) = j) \\
&\qquad \times \mathbb{P}(\tau_1 - \tau_{1,2} = z_1, \tau_2 - \tau_{1,2} = z_2 \mid \tau_{1,2} = m, M_1(m) = M_2(m) = j) \\
&\quad = \sum_{j \in \mathcal{S}} \mathbb{P}(\tau_{1,2} = m, M_1(m) = M_2(m) = j) \\
&\qquad \times \mathbb{P}(\tau_1 - \tau_{1,2} = z_1 \mid \tau_{1,2} = m, M_1(m) = j) \mathbb{P}(\tau_2 - \tau_{1,2} = z_2 \mid \tau_{1,2} = m, M_2(m) = j)\\
%&\qquad \times \mathbb{P}(\tau_2 - \tau_{1,2} = z_2 \mid \tau_{1,2} = m, M_2(m) = j) \\
&\quad = \sum_{j \in \mathcal{S}} (\bm{\alpha} \bm{P}^{m - 1} \bm{U})_j (\bm{Q}_1^{z_1 - 1} \bm{q}_1)_j (\bm{Q}_2^{z_2 - 1} \bm{q}_2)_j \\
&\quad = \bm{\alpha} \bm{P}^{m - 1} \bm{U} \left( (\bm{Q}_1^{z_1 - 1} \bm{q}_1) \odot (\bm{Q}_2^{z_2 - 1} \bm{q}_2) \right),
\end{align*}
which establishes \eqref{eq:psi1}. Finally, \eqref{eq:psik1} follows by replacing the initial distribution \( \bm{\alpha} \) with \( \bm{e}_i^\intercal \).
\end{proof}

Using Proposition \ref{th:mass-taua-taub}, one can readily show the following corollary.
\begin{corollary}\label{coro:pgf-taua-taub}
For $ \zeta_0 , \zeta_1, \zeta_2\in [0,1]$, the joint pgf of $(\tau_{1,2}, \tau_1 - \tau_{1,2}, \tau_2 - \tau_{1,2})$ is given by
\begin{align} \label{eq:pgfCDPH0}
\mathbb{E}&\left(\zeta_0^{\tau_{1,2}}  \zeta_1^{\tau_1-\tau_{1,2}} \zeta_2^{\tau_2-\tau_{1,2}}  \right)  = \sum_{j\in\mathcal{S}}(\bm{\alpha}(\bm{I}\zeta_0^{-1} -\bm{P})^{-1}\bm{U})_j((\bm{I}\zeta_1^{-1} -\bm{Q}_1)^{-1}\bm{q}_1)_j((\bm{I}\zeta_2^{-1} -\bm{Q}_2)^{-1}\bm{q}_2)_j.
\end{align}
Moreover, the joint pgf of $(\tau_1, \tau_2)$ is given by
\begin{align} \label{eq:pgfCDPH}
\mathbb{E}\left( \zeta_1^{\tau_1} \zeta_2^{\tau_2} \right) = \sum_{j\in\mathcal{S}}(\bm{\alpha}(\bm{I}(\zeta_1\zeta_2)^{-1} -\bm{P})^{-1}\bm{U})_j((\bm{I}\zeta_1^{-1} -\bm{Q}_1)^{-1}\bm{q}_1)_j((\bm{I}\zeta_2^{-1} -\bm{Q}_2)^{-1}\bm{q}_2)_j.
\end{align}
\end{corollary}

\begin{proof}
As $\psi$ is the joint pmf of $(\tau_{1,2},\tau_1-\tau_{1,2},\tau_2-\tau_{1,2})$, we have
\begin{align*} 
\mathbb{E}&\left(  \zeta_0^{\tau_{1,2}} \zeta_1^{\tau_1-\tau_{1,2}} \zeta_2^{\tau_2-\tau_{1,2}} \right)  = \sum_{z_0,z_1,z_2 \ge 1} \psi(z_0,z_1,z_2) \zeta_0^{z_0} \zeta_1^{z_1} \zeta_2^{z_2}\\
& = \sum_{z_0,z_1,z_2 \ge 1} \sum_{j\in\mathcal{S}} (\bm{\alpha} \bm{P}^{z_0-1}\bm{U})_j (\bm{Q}_1^{z_1-1}\bm{q}_1)_j (\bm{Q}_2^{z_2-1}\bm{q}_2)_j \zeta_0^{z_0} \zeta_1^{z_1} \zeta_2^{z_2} \\
& = \zeta_0 \zeta_1 \zeta_2 \sum_{j\in\mathcal{S}}\sum_{z_0,z_1,z_2 \ge 1} (\bm{\alpha} (\zeta_0\bm{P})^{z_0-1}\bm{U})_j ((\zeta_1\bm{Q}_1)^{z_1-1}\bm{q}_1)_j ((\zeta_2\bm{Q}_2)^{z_2-1}\bm{q}_2)_j\\
& = \zeta_0 \zeta_1 \zeta_2 \sum_{j\in\mathcal{S}} (\bm{\alpha}(\bm{I} - \zeta_0\bm{P})^{-1}\bm{U})_j((\bm{I} - \zeta_1 \bm{Q}_1)^{-1}\bm{q}_1)_j((\bm{I}- \zeta_2\bm{Q}_2)^{-1}\bm{q}_2)_j,%\\
%& = \sum_{j\in\mathcal{S}} (\bm{\alpha}(\bm{I}\zeta_0^{-1} -\bm{P})^{-1}\bm{U})_j((\bm{I}\zeta_1^{-1} -\bm{Q}_1)^{-1}\bm{q}_1)_j((\bm{I}\zeta_2^{-1} -\bm{Q}_2)^{-1}\bm{q}_2)_j,
\end{align*}
which yields \eqref{eq:pgfCDPH0}.  
Substituting $\zeta_0 = \zeta_1\zeta_2$ into \eqref{eq:pgfCDPH0} gives \eqref{eq:pgfCDPH}.
\end{proof}

Employing Corollary \ref{coro:pgf-taua-taub}, we are able to compute the cross moments of \( (\tau_1, \tau_2) \) via the following result regarding falling factorial moments. Here, for $n\in\mathbb{N}$ and a generic random variable $X$, we let $X^{[n]}=X(X-1)(X-2)\cdots (X-n+1)$ denote the falling factorial of $X$.
\begin{corollary}\label{coro:jointmoments1}
For $n_0, n_1,n_2\in\mathbb{N}$, the joint factorial moments of $(\tau_{1,2},\tau_1-\tau_{1,2},\tau_2-\tau_{1,2})$ are given by
\begin{align} \label{eq:jointmoments1}
\mathbb{E}\bigl(\tau_{1,2}^{[n_0]}\,(\tau_1-\tau_{1,2})^{[n_1]}\,(\tau_2-\tau_{1,2})^{[n_2]}\bigr) 
= \sum_{j\in\mathcal{S}}(\bm{\alpha}\widetilde{\bm{P}}(n_0)\bm{U})_j
(\widetilde{\bm{Q}}_1(n_1)\bm{q}_1)_j
(\widetilde{\bm{Q}}_2(n_2)\bm{q}_2)_j,
\end{align}
where
\begin{align*}
\widetilde{\bm{P}}(n_0)&:=\left. \frac{\partial^{n_0}}{\partial \zeta_1^{n_0}}(\bm{I}\zeta_0^{-1}-\bm{P})^{-1}\right|_{\zeta_0=1} = n_0!\bm{P}^{n_0-1}(\bm{I}-\bm{P})^{-n_0-1},\qquad n_0\in\mathbb{N}_+,\\
\widetilde{\bm{Q}}_1(n_1)&:=\left. \frac{\partial^{n_1}}{\partial \zeta_1^{n_1}}(\bm{I}\zeta_1^{-1}-\bm{Q}_1)^{-1}\right|_{\zeta_1=1} = n_1!\bm{Q}_1^{n_1-1}(\bm{I}-\bm{Q}_1)^{-n_1-1},\qquad n_1\in\mathbb{N}_+,\\
\widetilde{\bm{Q}}_2(n_2)&:=\left. \frac{\partial^{n_2}}{\partial \zeta_2^{n_2}}(\bm{I}\zeta_2^{-1}-\bm{Q}_2)^{-1}\right|_{\zeta_2=1} = n_2!\bm{Q}_2^{n_2-1}(\bm{I}-\bm{Q}_2)^{-n_2-1}, \qquad n_2\in\mathbb{N}_+,
\end{align*}
with $\widetilde{\bm{P}}(0)=(\bm{I}-\bm{P})^{-1}$, $\widetilde{\bm{Q}}_1(0)=(\bm{I}-\bm{Q}_1)^{-1}$ and $\widetilde{\bm{Q}}_2(0)=(\bm{I}-\bm{Q}_2)^{-1}$.
\end{corollary}

\begin{proof}
It is standard that the \((n_0, n_1, n_2)\)-th falling factorial moment is obtained from the joint pgf \eqref{eq:pgfCDPH0} by taking \(n_i\) derivatives with respect to \(\zeta_i\) for all $i=0,1,2$ and then evaluating at \(\zeta_0 = \zeta_1 = \zeta_2 = 1\) so that
\[
\mathbb{E}\bigl(\tau_{1,2}^{[n_0]}\,(\tau_1-\tau_{1,2})^{[n_1]}\,(\tau_2-\tau_{1,2})^{[n_2]}\bigr)
= \left.\frac{\partial^{n_0+n_1+n_2}}{\partial \zeta_0^{n_0}\partial \zeta_1^{n_1}\partial \zeta_2^{n_2}}\mathbb{E}(\zeta_0^{\tau_{1,2}} \zeta_1^{\tau_1-\tau_{1,2}} \zeta_2^{\tau_2-\tau_{1,2}} )\right|_{\zeta_0=1, \zeta_1=1, \zeta_2=1}.
\]
Exploiting the linearity of differentiation yields \eqref{eq:jointmoments1} as long as one can prove the formula
\begin{align*}%\label{eq:r-der-matrix}
\frac{\partial^r}{\partial \zeta_0^r}\zeta_0(\bm{I}-\bm{T}\zeta_0)^{-1}=r!\bm{T}^{r-1}(\bm{I}-\bm{T}\zeta_0)^{-r-1},\qquad r\in\mathbb{N}_+,
\end{align*}
where $\bm{T}$ can be \(\bm{P}\), \(\bm{Q}_1\) or \(\bm{Q}_2\). It can be seen from the proof of Theorem 1.2.69 of \cite{bladt2017matrix} that the above equation is indeed valid, concluding our proof.
\end{proof}
%First, note from \cite[pp. 35]{bladt2017matrix} that
%\[
%\frac{\partial}{\partial \zeta_0}(\bm{I}-\bm{P}\zeta_0)^{-1}=\bm{P}(\bm{I}-\bm{P}\zeta_0)^{-2},
%\]
%so that
%\begin{align*}
%\frac{\partial}{\partial \zeta_0}\zeta_0(\bm{I}-\bm{P}\zeta_0)^{-1} 
%&= (\bm{I}-\bm{P}\zeta_0)^{-1} + \zeta_0 \bm{P}(\bm{I}-\bm{P}\zeta_0)^{-2}\\
%&= \big[(\bm{I}-\bm{P}\zeta_0) + \zeta_0 \bm{P}\big](\bm{I}-\bm{P}\zeta_0)^{-2}\\
%&= (\bm{I}-\bm{P}\zeta_0)^{-2}.
%\end{align*}

%We proceed by induction. Suppose that \eqref{eq:r-der-matrix} holds for \(r-1 \in \{1, 2, \dots\}\). Then,
%\begin{align*}
%\frac{\partial^r}{\partial \zeta_0^r}\zeta_0(\bm{I}-\bm{P}\zeta_0)^{-1} 
%&= \frac{\partial}{\partial \zeta_0}\left(\frac{\partial^{r-1}}{\partial \zeta_0^{r-1}}\zeta_0(\bm{I}-\bm{P}\zeta_0)^{-1}\right)\\
%&= \frac{\partial}{\partial \zeta_0}\left((r-1)!\bm{P}^{r-2}(\bm{I}-\bm{P}\zeta_0)^{-r}\right)\\
%&= (r-1)!\bm{P}^{r-2}\frac{\partial}{\partial \zeta_0}\big((\bm{I}-\bm{P}\zeta_0)^{-r}\big)\\
%&= (r-1)!\bm{P}^{r-2}\big[r(\bm{I}-\bm{P}\zeta_0)^{-(r+1)} \bm{P}\big]\\
%&= r!\bm{P}^{r-1}(\bm{I}-\bm{P}\zeta_0)^{-r-1},
%\end{align*}
%confirming the equality \eqref{eq:r-der-matrix} and concluding the proof.
\begin{remark}
While Corollary \ref{coro:jointmoments1} does not explicitly provide the cross moments of \(\tau_1\) and \(\tau_2\), by writing $\tau_1=\tau_{1,2}+(\tau_1-\tau_{1,2})$ and $\tau_2=\tau_{1,2}+(\tau_2-\tau_{1,2})$ followed by applying binomial expansions of their powers, it can be readily seen that the afore-mentioned moments can be represented as linear combinations of the falling factorial moments in \eqref{eq:jointmoments1}. For instance, one finds 
\begin{align*}
\mathbb{E}\bigl(\tau_{1}^2\,\tau_{2}\bigr) = &~ \mathbb{E}(\tau_{1,2}^3) 
+ 2\mathbb{E}(\tau_{1,2}^2\,(\tau_1 - \tau_{1,2}))
+ \mathbb{E}(\tau_{1,2}\,(\tau_1 - \tau_{1,2})^2)
+ \mathbb{E}(\tau_{1,2}^2\,(\tau_2 - \tau_{1,2})) \\
& + 2\mathbb{E}(\tau_{1,2}\,(\tau_1 - \tau_{1,2})(\tau_2 - \tau_{1,2}))
+ \mathbb{E}((\tau_1 - \tau_{1,2})^2(\tau_2 - \tau_{1,2})).
\end{align*}
\end{remark}

Note that the support of \((\tau_1, \tau_2)\) is \((\mathbb{N} + 2) \times (\mathbb{N} + 2)\). For modeling purposes, it is often useful to consider discrete random variables that have a lattice different from \(\mathbb{N} + 2\). Thus, for \(k_1, k_2 \in \mathbb{R}\) and \(c_1, c_2 > 0\), one can define the vector \((N_1, N_2)\) where
\begin{align}\label{eq:N1N2Def}
N_1 := c_1 (\tau_1-2) + k_1, \qquad N_2 := c_2 (\tau_2 -2) + k_2.
\end{align}
We then say in a broad sense (see Remark \ref{RemarkCDPH}) that the bivariate random variable \((N_1, N_2)\) follows a bivariate CDPH distribution with support on \((c_1 \mathbb{N} + k_1) \times (c_2 \mathbb{N} + k_2)\) and parameters \((\bm{\alpha}, \bm{P}, \bm{U}, \bm{Q}_1, \bm{Q}_2)\). Employing Proposition \ref{th:mass-taua-taub}, we obtain the following immediate result.

\begin{corollary}
Let \((N_1, N_2)\) follow a bivariate \(\operatorname{CDPH}\) distribution with support \((c_1 \mathbb{N} + k_1) \times (c_2 \mathbb{N} + k_2)\) and parameters \((\bm{\alpha}, \bm{P}, \bm{U}, \bm{Q}_1, \bm{Q}_2)\). Then, the joint pmf of \((N_1, N_2)\) for \(x_1 \in c_1 \mathbb{N} + k_1\) and \(x_2 \in c_2 \mathbb{N} + k_2\), namely $f_{N_1, N_2}(x_1, x_2):=\mathbb{P}(N_1 = x_1,\, N_2 = x_2)$, is given by
\begin{align*}
f_{N_1, N_2}(x_1, x_2) = f_{\tau_1,\tau_2}\left( \frac{x_1 - k_1}{c_1}+ 2,\; \frac{x_2 - k_2}{c_2}+ 2 \right),
\end{align*}
where \(f_{\tau_1,\tau_2}(n_1, n_2) = \mathbb{P}(\tau_1 = n_1,\, \tau_2 = n_2)\).
\end{corollary}

\subsection{Relationships with other bivariate DPH distributions}\label{subsec:CDPHProperty}

Let us first note that the (bivariate version of) the so-called \(\operatorname{mDPH}\) class considered in \cite{bladt2023robust} is a subclass of our proposed \(\operatorname{CDPH}\). Indeed, if we let \(\bm{\beta}\) be a probability row vector and \(\bm{Q}\) a subprobability matrix, then choosing the parameters \(\bm{\alpha} = (1)\), \(\bm{P} = (0)\), \(\bm{U} = \bm{\beta}\), and \(\bm{Q}_1 = \bm{Q}_2 = \bm{Q}\) for a \(\operatorname{CDPH}\) over the lattice \((\mathbb{N} + 1)\times(\mathbb{N} + 1)\) (i.e. $c_1=c_2=1$ and $k_1=k_2=-1$) yields an \(\operatorname{mDPH}\) distribution. This distribution arises from starting two terminating processes in the same state according to the initial probability vector \(\bm{\beta}\) then and evolving them independently from there onwards according to the subprobability matrix \(\bm{Q}\).

In turn, the \(\operatorname{CDPH}\) with support in \((\mathbb{N}+2) \times (\mathbb{N}+2)\) forms a subclass of the \(\operatorname{MDPH}^*\) class introduced in Chapter 5.2 of \cite{navarro2019order}. These distributions accumulate integer-valued rewards across each occupation epoch within a shared Markov chain, rather than bifurcating into independent processes after a common shock component concludes. Specifically, let us consider a partitioned state-space \(\mathcal{E} \cup (\mathcal{S} \times \mathcal{S}) \cup (\{\partial_1\} \times \mathcal{S}) \cup (\mathcal{S} \times \{\partial_2\})\) and a Markov chain \(\gamma = \{\gamma(n)\}_{n \in \mathbb{N}}\) defined over it. The chain evolves within \(\mathcal{E}\) while the common shock is ongoing, then jumps to \(\mathcal{S} \times \mathcal{S}\) (here we take the lexicographic order of \(\mathcal{S} \times \mathcal{S}\)) while the independent components are active, and finally either terminates, jumps to \(\{\partial_1\} \times \mathcal{S}\), or jumps to \(\mathcal{S} \times \{\partial_2\}\). These transitions occur depending on whether the independent components terminate at the same time, the first coordinate terminates before the second, or the second coordinate terminates before the first. Thus, we define the bivariate random vector \((\tau_1,\tau_2)\) as
\begin{align}\label{eq:tau_as_Kulkarni}
\tau_1:=\sum_{\ell=0}^{\sigma-1}\mathds{1}_{\gamma(\ell)\in \mathcal{E}\cup(\mathcal{S}\times\mathcal{S})\cup(\mathcal{S}\times\{\partial_2\})},\qquad 
\tau_2:=\sum_{\ell=0}^{\sigma-1}\mathds{1}_{\gamma(\ell)\in \mathcal{E}\cup(\mathcal{S}\times\mathcal{S})\cup(\{\partial_1\} \times \mathcal{S})},
\end{align}
where $\sigma$ corresponds to the termination time of $\gamma$ so that $\sigma:=\inf\{ n \in \mathbb{N}_+ : \gamma(n) \notin \mathcal{E} \cup (\mathcal{S} \times \mathcal{S}) \cup (\{\partial_1\} \times \mathcal{S}) \cup (\mathcal{S} \times \{\partial_2\}) \}$. The above representation of the bivariate vector \((\tau_1,\tau_2)\) is already in the form of Equation (5.5) of \cite{navarro2019order}. It remains to specify the transition probability of the terminating Markov chain $\gamma$.

To mimic the distributional properties of the original processes \(M_1\) and \(M_2\) embedded in the Markov chain \(\gamma\), we take the block-partitioned parameters \((\bm{\alpha},\bm{0},\bm{0},\bm{0})\), and the subprobability matrix
\begin{align}\label{eq:maximum-transition}
\bm{P}_{\mathrm{max}}:= \begin{pmatrix}
\bm{P} & \bm{U}^* & \bm{0} & \bm{0}\\
\bm{0} &  \bm{Q}_1 \otimes \bm{Q}_2 & \bm{q}_1 \otimes \bm{Q}_2 & \bm{Q}_1 \otimes \bm{q}_2 \\ 
\bm{0} & \bm{0} & \bm{Q}_2 & \bm{0}\\
\bm{0} & \bm{0} & \bm{0} & \bm{Q}_1
\end{pmatrix},
\end{align}
where \(\bm{U}^* = \{u_{i,(j_1,j_2)}^*\}_{i \in \mathcal{E}, (j_1,j_2) \in \mathcal{S} \times \mathcal{S}}\) is defined by \(u_{i,(j_1,j_2)}^* := u_{ij_1} \delta_{j_1,j_2}\), with \(\delta_{j_1,j_2}\) denoting the Kronecker delta and $u_{ij}$ the $(i,j)$-th element of $\bm{U}$, and \(\otimes\) represents the Kronecker product between matrices. Then, $\sigma$ is $\operatorname{DPH}((\bm{\alpha},\bm{0},\bm{0},\bm{0}),\bm{P}_{\mathrm{max}})$. Note that the resulting parameters are similar to those used to compute the distribution of the maximum between two discrete phase-type distributions \cite[Theorem 1.2.67]{bladt2017matrix}. The derivation of these parameters for the process \(\gamma\) is as follows:
\begin{itemize}
    \item The process evolves in \(\mathcal{E}\) according to \(\bm{P}\), mimicking the common shock component \(\tau_{1,2}\), eventually jumping from a state \(i \in \mathcal{E}\) to a state \((j,j) \in \mathcal{S} \times \mathcal{S}\) with probability \(u_{ij}\).
    \item It then evolves in \(\mathcal{S} \times \mathcal{S}\) according to \(\bm{Q}_1 \otimes \bm{Q}_2\) such that a transition from \((i_1, i_2) \in \mathcal{S} \times \mathcal{S}\) to \((j_1, j_2) \in \mathcal{S} \times \mathcal{S}\) occurs with probability %\((\bm{Q}_1)_{i_1,j_1} (\bm{Q}_2)_{i_2,j_2}\)
    $q_{1,i_1j_1}\,q_{2,i_2j_2}$ (where, for $k=1,2$, we use $q_{k,ij}$ to denote the $(i,j)$-th element of $\bm{Q}_k$). Then, either both components terminate simultaneously according to the exit column vector \(\bm{q}_1 \otimes \bm{q}_2\); the first terminates while the second remains active, moving to \(\{\partial_1\} \times \mathcal{S}\); or the second terminates while the first remains active, moving to \(\mathcal{S} \times \{\partial_2\}\).
    \item If the process jumps to \(\{\partial_1\} \times \mathcal{S}\) or \(\mathcal{S} \times \{\partial_2\}\), it continues evolving according to the matrix \(\bm{Q}_2\) or \(\bm{Q}_1\), respectively.
\end{itemize}

The above results are summarized in the following proposition.
\begin{proposition}\label{th:ClassProperty}
Concerning bivariate $\operatorname{DPH}$ distributions, our proposed $\operatorname{CDPH}$ class satisfies the relationships
\begin{align*}
&{\normalfont \text{mDPH $\subset$ CDPH (with support $(\mathbb{N}+1) \times (\mathbb{N}+1)$)}},\\
&{\normalfont \text{CDPH (with support $(\mathbb{N}+2) \times (\mathbb{N}+2)$) $\subset$ $\operatorname{MDPH}^*$}},
\end{align*}
where $\operatorname{mDPH}$ is the class introduced by \cite{bladt2023robust} and $\operatorname{MDPH}^*$ is the class introduced by \cite{navarro2019order}.
\end{proposition}

It is important to emphasize that, although the $\operatorname{MDPH}^*$ class in \cite{navarro2019order} is a larger class than our CDPH class on  $(\mathbb{N}+2) \times (\mathbb{N}+2)$, a CDPH distributed bivariate random vector admits explicit joint pmf as in Proposition \ref{th:mass-taua-taub} while the same is not true for the $\operatorname{MDPH}^*$ class.

\begin{remark}\label{RemarkCDPH}
Since the general \(\operatorname{CDPH}\) distributed random vector \((N_1, N_2)\) in \eqref{eq:N1N2Def} with support on \((c_1 \mathbb{N} + k_1) \times (c_2 \mathbb{N} + k_2)\) is merely linear transformation of its basic version \((\tau_1, \tau_2)\) which has support on \((\mathbb{N}+2) \times (\mathbb{N}+2)\), from now we shall simply use $\operatorname{CDPH}$ to denote the class with support on \((\mathbb{N}+2) \times (\mathbb{N}+2)\), and say that \((\tau_1, \tau_2)\) follows a $\operatorname{CDPH}(\bfalp, \bm{P}, \bm{U}, \bm{Q}_1, \bm{Q}_2)$ distribution.
\end{remark}

\subsection{Further distributional and closure properties}\label{subsec:Operations}

Below, we explore further properties of the $\operatorname{CDPH}$ class. In particular, we investigate certain operations on the vector $(\tau_1, \tau_2)$ that belong to the class of $\operatorname{DPH}$ or $\operatorname{CDPH}$ distributions.

\subsubsection*{\emph{Minimum and maximum between $\tau_1$ and $\tau_2$}}
Consider again the Markov chain $\gamma$ driven by \eqref{eq:maximum-transition}, with initial vector \((\bm{\alpha}, \bm{0}, \bm{0}, \bm{0})\), state-space \(\mathcal{E} \cup (\mathcal{S} \times \mathcal{S}) \cup (\{\partial_1\} \times \mathcal{S}) \cup (\mathcal{S} \times \{\partial_2\})\) and termination time $\sigma$. One can easily see that $\sigma$ indeed corresponds to the maximum $\tau_{\max}:=\max\{\tau_1, \tau_2\}$ of $\tau_1$ and $\tau_2$. Moreover, the minimum $\tau_{\min}:=\min\{\tau_1, \tau_2\}$ is simply the exit time of the Markov chain $\gamma$ from the set $\mathcal{E} \cup (\mathcal{S} \times \mathcal{S})$. Defining the matrix
\begin{align}\label{eq:minimum-transition}
\bm{P}_{\mathrm{min}}:= \begin{pmatrix}
\bm{P} & \bm{U}^* \\
\bm{0} &  \bm{Q}_1 \otimes \bm{Q}_2 \\ 
\end{pmatrix},
\end{align}
the next proposition follows as a direct consequence.
\begin{proposition}
Given a $\operatorname{CDPH}(\bfalp, \bm{P}, \bm{U}, \bm{Q}_1, \bm{Q}_2)$ random vector $(\tau_1,\tau_2)$, the maximum $\tau_{\max}$ follows a $\operatorname{DPH}((\bm{\alpha},\bm{0},\bm{0},\bm{0}),\bm{P}_{\mathrm{max}})$ distribution and the minimum $\tau_{\min}$ follows a \linebreak $\operatorname{DPH}((\bm{\alpha},\bm{0}),\bm{P}_{\mathrm{min}})$ distribution, where $\bm{P}_{\mathrm{max}}$ and $\bm{P}_{\mathrm{min}}$ are given by \eqref{eq:maximum-transition} and \eqref{eq:minimum-transition}, respectively.
\end{proposition}

%To track the first termination time between $\tau_1$ and $\tau_2$, we define
%\begin{align}\label{eq:tau-min}
%\tau_{\min} = \sum_{\ell=0}^{\sigma-1}\mathds{1}_{\gamma(\ell) \in \mathcal{E} \cup (\mathcal{S} \times \mathcal{S})}.
%\end{align}
%Indeed, the sum on the right-hand side of \eqref{eq:tau-min} counts the number of steps during which both coordinates have not yet terminated. Thus, $\tau_{\min}$ has the same distribution as the minimum between $\tau_1$ and $\tau_2$. Similarly, the random variable
%\begin{align*}
%\tau_{\max} = \sum_{\ell=0}^{\sigma-1}\mathds{1}_{\gamma(\ell) \in \mathcal{E} \cup (\mathcal{S} \times \mathcal{S}) \cup (\{\partial_1\} \times \mathcal{S}) \cup (\mathcal{S} \times \{\partial_2\})}
%\end{align*}
%counts the total number of steps until both coordinates terminate, yielding $\tau_{\max} \stackrel{d}{=} \max\{\tau_1, \tau_2\}$. Since both $\tau_{\min}$ and $\tau_{\max}$ have an integer reward representation of a terminating Markov chain, it follows by \cite[Theorem 5.2]{navarro2019order} that they follow a $\operatorname{DPH}$ distribution.

\subsubsection*{\emph{Aggregation of coordinates \(\tau_1 + \tau_2\)}}
When considering the sum of the coordinates of $(\tau_1,\tau_2)$, it is helpful to represent it as $\tau_1 + \tau_2 = 2\min\{\tau_1, \tau_2\} + (\max\{\tau_1, \tau_2\} - \min\{\tau_1, \tau_2\})$. That is, $\tau_1 + \tau_2$ counts twice the steps in which both $\tau_1$ and $\tau_2$ have not terminated, and then adds the steps during which the larger value exceeds the smaller one. Consequently, we have
\begin{align}\label{eq:aggregation}
\tau_1+\tau_2 = \sum_{\ell=0}^{\sigma-1}\left(2 \cdot \mathds{1}_{\gamma(\ell) \in \mathcal{E} \cup (\mathcal{S} \times \mathcal{S})} + \mathds{1}_{\gamma(\ell) \in (\{\partial_1\} \times \mathcal{S}) \cup (\mathcal{S} \times \{\partial_2\})}\right).
\end{align}
We have the following proposition.
\begin{proposition}
Given a $\operatorname{CDPH}(\bfalp, \bm{P}, \bm{U}, \bm{Q}_1, \bm{Q}_2)$ random vector $(\tau_1,\tau_2)$, the sum $\tau_1+\tau_2$ follows a $\operatorname{DPH}$ distribution.
\end{proposition}
\begin{proof}
With the representation \eqref{eq:aggregation} for the random variable $\tau_1+\tau_2$, it follows from Theorem 5.2 in \cite{navarro2019order} that the distribution of $\tau_1+\tau_2$ is a mixture of a probability mass at zero and a $\operatorname{DPH}$ distribution. However, both $\tau_1$ and $\tau_2$ are at least one from their definitions, implying that the afore-mentioned probability mass at zero must be non-existent and hence $\tau_1+\tau_2$ follows a $\operatorname{DPH}$ distribution. The construction of the parameters of such a $\operatorname{DPH}$ distribution can be found in Chapter 5.1.1.2 in \cite{navarro2019order} and is omitted for brevity.
\end{proof}

\subsubsection*{\emph{Mixture of CDPH random vectors}}
For each $i\in \{1, 2, \dots, m\}$, suppose that the random vector $(\tau_1(i),\tau_2(i))$ is CDPH distributed. Let a mixture be formed where the component $(\tau_1(i),\tau_2(i))$ is chosen with probability $w_i > 0$. The following proposition shows that the class of CDPH distributions is closed under finite mixture.
\begin{proposition}
Suppose that a mixture is formed such that a bivariate random vector $(\tau_{1,\mathrm{mix}},\tau_{2,\mathrm{mix}})$ is chosen to be $(\tau_1(i),\tau_2(i))$ with probability $w_i > 0$, such that $(\tau_1(i),\tau_2(i))$ is $\operatorname{CDPH}(\bfalp(i), \bm{P}(i), \bm{U}(i), \bm{Q}_1(i), \bm{Q}_2(i))$ distributed and $\sum_{i=1}^m w_i = 1$. Then, $(\tau_{1,\mathrm{mix}},\tau_{2,\mathrm{mix}})$ follows a $\operatorname{CDPH}(\bfalp_{\mathrm{mix}}, \bm{P}_{\mathrm{mix}}, \bm{U}_{\mathrm{mix}}, \bm{Q}_{1,\mathrm{mix}}, \bm{Q}_{2,\mathrm{mix}})$ distribution, where the parameters are given by
\[
\bfalp_{\mathrm{mix}} := \begin{pmatrix}
w_1 \bfalp(1) & w_2 \bfalp(2) & \cdots & w_m \bfalp(m)
\end{pmatrix},
\]
\[
\bm{P}_{\mathrm{mix}} := \begin{pmatrix}
\bm{P}(1) & \bm{0} & \cdots & \bm{0} \\
\bm{0} & \bm{P}(2) & \cdots & \bm{0} \\
\vdots & \vdots & \ddots & \vdots \\
\bm{0} & \bm{0} & \cdots & \bm{P}(m)
\end{pmatrix},
\quad
\bm{U}_{\mathrm{mix}} := \begin{pmatrix}
\bm{U}(1) & \bm{0} & \cdots & \bm{0} \\
\bm{0} & \bm{U}(2) & \cdots & \bm{0} \\
\vdots & \vdots & \ddots & \vdots \\
\bm{0} & \bm{0} & \cdots & \bm{U}(m)
\end{pmatrix},
\]
\[
\bm{Q}_{1,\mathrm{mix}} := \begin{pmatrix}
\bm{Q}_1(1) & \bm{0} & \cdots & \bm{0} \\
\bm{0} & \bm{Q}_1(2) & \cdots & \bm{0} \\
\vdots & \vdots & \ddots & \vdots \\
\bm{0} & \bm{0} & \cdots & \bm{Q}_1(m)
\end{pmatrix},
\quad
\bm{Q}_{2,\mathrm{mix}} := \begin{pmatrix}
\bm{Q}_2(1) & \bm{0} & \cdots & \bm{0} \\
\bm{0} & \bm{Q}_2(2) & \cdots & \bm{0} \\
\vdots & \vdots & \ddots & \vdots \\
\bm{0} & \bm{0} & \cdots & \bm{Q}_2(m)
\end{pmatrix}.
\]
\end{proposition}
\begin{proof}
Indeed, the initial distribution \(\bfalp_{\mathrm{mix}}\) is simply a weighted combination of the initial probabilities \(\{\bfalp(i)\}_{i=1}^m\) of the component CDPH distributions, with weight \(w_i\) assigned to $\bfalp_i$. The transition matrix \(\bm{P}_{\mathrm{mix}}\) is a block-diagonal matrix, where each block corresponds to the transition matrix \(\bm{P}(i)\) of the \(i\)-th component. Similarly, \(\bm{U}_{\mathrm{mix}}, \bm{Q}_{1,\mathrm{mix}},\) and \(\bm{Q}_{2,\mathrm{mix}}\) are block-diagonal matrices, where each block corresponds to \(\bm{U}(i)\), \(\bm{Q}_1(i)\), and \(\bm{Q}_2(i)\), respectively. Such block-diagonal structure ensures that there is no interaction among the components (i.e. $(\tau_1(i),\tau_2(i))$ across different $i$'s) once a specific component is chosen according to the mixing weights \(\{w_i\}_{i=1}^m\). This construction guarantees that the resulting parameters describe a valid CDPH distribution for $(\tau_{1,\mathrm{mix}},\tau_{2,\mathrm{mix}})$.
\end{proof}

\subsubsection*{\emph{Sum of CDPH random vectors}}

Suppose that the random vector $(\tau_1(i),\tau_2(i))$ is CDPH distributed for each $i\in \{1, 2, \dots, m\}$, and the random vectors $\{(\tau_1(i),\tau_2(i))\}_{i=1}^m$ are independent. Then one could be interested in the joint distribution of the sum $(\tau_{1,\mathrm{sum}},\tau_{2,\mathrm{sum}}):=(\sum_{i=1}^m\tau_1(i),\sum_{i=1}^m\tau_2(i))$. Indeed, this follows a CDPH distribution. In the next proposition, we provide a proof for the case $m=2$ and explicitly state the resulting parameters. %In terms of notation, we shall add `$(i)$' to a symbol to emphasize that it belongs to the vector $(\tau_1(i),\tau_2(i))$.
\begin{proposition}\label{prop:sumCDPH}
Consider the bivariate random vector $(\tau_{1,\mathrm{sum}},\tau_{2,\mathrm{sum}})=(\tau_1(1)+\tau_1(2),\tau_2(1)+\tau_2(2))$, where $(\tau_1(i),\tau_2(i))$ is $\operatorname{CDPH}(\bfalp(i), \bm{P}(i), \bm{U}(i), \bm{Q}_1(i), \bm{Q}_2(i))$ distributed for $i\in\{1,2\}$, and the random vectors $(\tau_1(1),\tau_2(1))$  and $(\tau_1(2),\tau_2(2))$ are independent. Then, $(\tau_{1,\mathrm{sum}},\tau_{2,\mathrm{sum}})$ follows a $\operatorname{CDPH}(\bfalp_{\mathrm{sum}}, \bm{P}_{\mathrm{sum}}, \bm{U}_{\mathrm{sum}}, \bm{Q}_{1,\mathrm{sum}}, \bm{Q}_{2,\mathrm{sum}})$ distribution, where the parameters are given by
\begin{align}\label{eq:alphasum}
\bfalp_{\mathrm{sum}} := (\bfalp(1),\bm{0}),
\end{align}
\begin{align}\label{eq:PUsum}
\bm{P}_{\mathrm{sum}} := \begin{pmatrix}
\bm{P}(1) & \bm{U}(1) \otimes \bfalp(2) \\
\bm{0} & \bm{I} \otimes \bm{P}(2)
\end{pmatrix},
\quad
\bm{U}_{\mathrm{sum}} := \begin{pmatrix}
\bm{0} & \bm{0} \\
\bm{I} \otimes \bm{U}(2) & \bm{0}
\end{pmatrix},
\end{align}
\begin{align}\label{eq:Q1Q2sum}
\bm{Q}_{1,\mathrm{sum}} := \begin{pmatrix}
\bm{Q}_1(1) \otimes \bm{I} & \bm{q}_1(1) \otimes \bm{I} \\
\bm{0} & \bm{Q}_1(2)
\end{pmatrix},
\quad
\bm{Q}_{2,\mathrm{sum}} := \begin{pmatrix}
\bm{I} \otimes \bm{Q}_2(1) & \bm{I} \otimes \bm{q}_2(1) \\
\bm{0} & \bm{Q}_2(2)
\end{pmatrix},
\end{align}
with \( \bm{q}_{1}(1) = \bm{1} - \bm{Q}_{1}(1) \bm{1} \) and \( \bm{q}_{2}(1) = \bm{1} - \bm{Q}_{2}(1) \bm{1} \). The state space corresponding to the transitions specified by $\bm{P}_{\mathrm{sum}}$ is $\mathcal{E}_{\mathrm{sum}} := \mathcal{E}(1) \cup (\mathcal{S}(1) \times \mathcal{E}(2))$, and that for the transitions specified by $\bm{Q}_{1,\mathrm{sum}}$ and $\bm{Q}_{2,\mathrm{sum}}$ is $(\mathcal{S}(1) \times \mathcal{S}(2)) \cup \mathcal{S}(2)$. Here, with the obvious notation, $\mathcal{E}(i)\cup\mathcal{S}(i)$ is the state space of the underlying terminating Markov chains that define the vector $(\tau_1(i),\tau_2(i))$.
\end{proposition}
\begin{proof}
In order to explain the key ideas behind the proof, for $i,k\in\{1,2\}$ it will be helpful to write $\tau_k(i) = \tau_{1,2}(i)+(\tau_k(i)-\tau_{1,2}(i))$. Here the random variable $\tau_{1,2}(i)$ is the common shock component for the pair $(\tau_1(i),\tau_2(i))$. We can thus decompose each coordinate of the random vector $(\tau_{1,\mathrm{sum}},\tau_{2,\mathrm{sum}})=(\tau_1(1)+\tau_1(2),\tau_2(1)+\tau_2(2))$ as
\begin{align*}
\tau_{1,\mathrm{sum}} &= (\tau_{1,2}(1)+\tau_{1,2}(2)) + (\tau_1(1)-\tau_{1,2}(1)) + (\tau_1(2)-\tau_{1,2}(2)),\\ %\label{eq:sumconstruction1}\\
\tau_{2,\mathrm{sum}} &= (\tau_{1,2}(1)+\tau_{1,2}(2)) + (\tau_2(1)-\tau_{1,2}(1)) + (\tau_2(2)-\tau_{1,2}(2)). %\label{eq:sumconstruction2}
\end{align*}
To show that $(\tau_{1,\mathrm{sum}},\tau_{2,\mathrm{sum}})$ follows a CDPH distribution, we note from the above representation that $\tau_{1,2}(1)+\tau_{1,2}(2)$ can be regarded as the common shock component from which we can construct $\bm{P}_{\mathrm{sum}}$. After $\tau_{1,2}(1)$ and $\tau_{1,2}(2)$ are realized and given the states in $\mathcal{S}(1)$ and $\mathcal{S}(2)$ entered upon exiting $\mathcal{E}(1)$ and $\mathcal{E}(2)$, the variables $(\tau_1(1)-\tau_{1,2}(1)) + (\tau_1(2)-\tau_{1,2}(2))$ and $(\tau_2(1)-\tau_{1,2}(1)) + (\tau_2(2)-\tau_{1,2}(2))$ are conditionally independent: such observation can be used to construct $\bm{Q}_{1,\mathrm{sum}}$ and $\bm{Q}_{2,\mathrm{sum}}$. The construction of the matrix parameters of the CDPH representation of $(\tau_{1,\mathrm{sum}},\tau_{2,\mathrm{sum}})$ is described as follows:
\begin{itemize}
    \item We start by considering the part $\tau_{1,2}(1)$ that belongs to the common shock $\tau_{1,2}(1)+\tau_{1,2}(2)$. The process underlying $(\tau_{1,\mathrm{sum}},\tau_{2,\mathrm{sum}})$ starts in $\mathcal{E}(1)$ with the initial probability $\bfalp(1)$ (which explains \eqref{eq:alphasum}) and then transitions within $\mathcal{E}(1)$ according to $\bm{P}(1)$. Upon exiting $\mathcal{E}(1)$, the variable $\tau_{1,2}(1)$ is realized and the process underlying $(\tau_1(1),\tau_2(1))$ enters $\mathcal{S}(1)$ according to $\bm{U}(1)$. While one can start counting the remaining common shock component $\tau_{1,2}(2)$ using the initial probability $\bfalp(2)$ and transition probability matrix $\bm{P}(2)$ within $\mathcal{E}(2)$, it is important to also keep track of the state in $\mathcal{S}(1)$ because this will have an impact on the distributions of $\tau_1(1)-\tau_{1,2}(1)$ and $\tau_2(1)-\tau_{1,2}(1)$. These explain the block matrix representation of $\bm{P}_{\mathrm{sum}}$ in \eqref{eq:PUsum}. The submatrix $\bm{I} \otimes \bm{U}(2)$ of $\bm{U}_{\mathrm{sum}}$ in \eqref{eq:PUsum} subsequently takes care of the termination of $\tau_{1,2}(2)$ as the process exits $\mathcal{S}(1) \times \mathcal{E}(2)$ and enters $\mathcal{S}(1) \times \mathcal{S}(2)$.

    \item Once the underlying process of $(\tau_{1,\mathrm{sum}},\tau_{2,\mathrm{sum}})$ has entered $\mathcal{S}(1) \times \mathcal{S}(2)$, the clocks for $\tau_1(1)-\tau_{1,2}(1)$ and $\tau_2(1)-\tau_{1,2}(1)$ start ticking independently within the set $\mathcal{S}(1)$ according to the transition matrices $\bm{Q}_{1}(1)$ and $\bm{Q}_{2}(1)$, respectively, while keeping the state in $\mathcal{S}(2)$ fixed. When $\tau_1(1)-\tau_{1,2}(1)$ (resp. $\tau_2(1)-\tau_{1,2}(1)$) terminates according to the probability vector $\bm{q}_1(1)$ (resp. $\bm{q}_2(1)$), one no longer needs to keep track of the state space $\mathcal{S}(1)$, and then we move on to consider $\tau_1(2)-\tau_{1,2}(2)$ (resp. $\tau_2(2)-\tau_{1,2}(2)$) as the process evolves within $\mathcal{S}(2)$ according to $\bm{Q}_{1}(2)$ (resp. $\bm{Q}_{2}(2)$) until it gets absorbed. Such arguments lead to the specification of the matrices $\bm{Q}_{1,\mathrm{sum}}$ and $\bm{Q}_{2,\mathrm{sum}}$ in \eqref{eq:Q1Q2sum}.

\end{itemize}
The above construction ensures that $(\tau_{1,\mathrm{sum}},\tau_{2,\mathrm{sum}})$ is CDPH distributed, and the proof is complete.
\end{proof}
It is clear that $(\tau_{1,\mathrm{sum}},\tau_{2,\mathrm{sum}})=(\sum_{i=1}^m\tau_1(i),\sum_{i=1}^m\tau_2(i))$ for values of $m$ greater than two still follows a CDPH distribution by summing the CDPH bivariate random vectors one at a time and using Proposition \ref{prop:sumCDPH} repeatedly.

\section{Compound sums under bivariate CDPH claim count}\label{sec:random-sums}

Having defined the class of bivariate CDPH distributions, we now turn our attention to a class of bivariate random variables \( (Y_1, Y_2) \) with a random sum representation
\begin{align}\label{eq:Y1Y2def}
Y_1 := \sum_{\ell=1}^{\tau_1} X_{1,\ell}, \qquad Y_2 := \sum_{\ell=1}^{\tau_2} X_{2,\ell},
\end{align}
where, for each coordinate \( k \in \{1, 2\} \), we aggregate a random number of positive variables \( X_{k,1}, X_{k,2}, X_{k,3}, \dots \), with the total number \( \tau_k \in \mathbb{N} \) drawn from a bivariate CDPH distribution. In actuarial science, $(\tau_1,\tau_2)$ can be interpreted as bivariate claim count for correlated business lines. For \( k \in \{1, 2\} \), if \( X_{k,1}, X_{k,2}, X_{k,3}, \dots \) are the individual claim amounts specific to line $k$, then the compound sum $Y_k$ corresponds to the aggregate loss for line $k$. Before analyzing the distribution of \((Y_1, Y_2)\) in detail, we briefly outline the assumptions we make regarding the sequence \( \{ (X_{1,\ell}, X_{2,\ell}) \}_{\ell \in \mathbb{N}_+} \). Firstly, we assume that \( \{ (X_{1,\ell}, X_{2,\ell}) \}_{\ell \in \mathbb{N}_+} \) is a sequence of independent and identically distributed bivariate random vectors following the same distribution as a reference random vector \( (X_1, X_2) \). Secondly, the sequence \( \{ (X_{1,\ell}, X_{2,\ell}) \}_{\ell \in \mathbb{N}_+} \) is independent of \((\tau_1, \tau_2)\). While the CDPH distributed \((\tau_1, \tau_2)\) has support on \((\mathbb{N}+2) \times (\mathbb{N}+2)\), the cases where the counting variables are replaced by more general CDPH distributed \((N_1, N_2)\) with support on \((\mathbb{N}+k_1) \times (\mathbb{N}+k_2)\) for $k_1,k_2\in\mathbb{N}$ can be addressed by modifying our arguments with minimal technical effort (see Remark \ref{RemarkCDPH}).

We shall consider two scenarios for \( X_1 \) and \( X_2 \): the case where they are independent, discussed in Subsection \ref{subsec:independent-arrivals}, and the case where they are dependent, addressed in Subsection \ref{subsec:dependent-arrivals}.

\subsection{Independent summands}\label{subsec:independent-arrivals}

Suppose that \( X_1 \) is independent of \( X_2 \). Below, we provide information about the distribution of \( (Y_1, Y_2) \) under this assumption.

\begin{proposition}\label{th:bivariateLaplace1} 
For $\theta_1,\theta_2\ge0$, the joint Laplace transform $\mathcal{L}_{Y_1,Y_2}(\theta_1,\theta_2):=\mathbb{E}\left( \exp(-\theta_1 Y_1 - \theta_2 Y_2) \right)$ of the bivariate random sum in \eqref{eq:Y1Y2def} is given by
\begin{align*}
\mathcal{L}_{Y_1,Y_2}(\theta_1,\theta_2) = \bm{\alpha} \left( \bm{I} a^{-1}_{1,\theta_1} a^{-1}_{2,\theta_2} - \bm{P} \right)^{-1} \bm{U} \left( \left( \left( \bm{I} a^{-1}_{1,\theta_1} - \bm{Q}_1 \right)^{-1} \bm{q}_1 \right) \odot \left( \left( \bm{I} a^{-1}_{2,\theta_2} - \bm{Q}_2 \right)^{-1} \bm{q}_2 \right) \right),
\end{align*}
where \( a_{k,\theta_k} := \mathbb{E}\left( e^{-\theta_k X_{k}} \right) \) is the Laplace transform of $X_{k}$ for \( k \in \{1, 2\} \).
\end{proposition}

\begin{proof}
By conditioning on $(\tau_1,\tau_2)$, we have
\begin{align*}
\mathcal{L}_{Y_1,Y_2}(\theta_1,\theta_2) &= \mathbb{E}\left( \exp\left( -\theta_1 \sum_{\ell=1}^{\tau_1} X_{1,\ell} - \theta_2 \sum_{\ell=1}^{\tau_2} X_{2,\ell} \right) \right)\\
&= \mathbb{E}\left( \mathbb{E}\left( \left. \exp\left( -\theta_1 \sum_{\ell=1}^{\tau_1} X_{1,\ell} - \theta_2 \sum_{\ell=1}^{\tau_2} X_{2,\ell} \right) \right |\tau_1, \tau_2 \right) \right)\\
&= \mathbb{E}\left( \left[ \prod_{\ell=1}^{\tau_1} \mathbb{E}\left( e^{-\theta_1 X_{1,\ell}} \right) \right] \left[ \prod_{\ell=1}^{\tau_2} \mathbb{E}\left( e^{-\theta_2 X_{2,\ell}} \right) \right] \right) \\
&= \mathbb{E}\left( a_{1,\theta_1}^{\tau_1} \, a_{2,\theta_2}^{\tau_2} \right).
\end{align*}

The above expression is simply the joint pgf in \eqref{eq:pgfCDPH} evaluated at the arguments $(\zeta_1,\zeta_2)=( a_{1,\theta_1}, a_{2,\theta_2})$. Since $a_{1,\theta_1}$ and $a_{2,\theta_2}$ are the Laplace transforms of positive random variables, one must have $a_{1,\theta_1},a_{2,\theta_2}\in[0,1]$ for $\theta_1,\theta_2\ge0$. This yields the desired result.
\end{proof}

Proposition \ref{th:bivariateLaplace1} provides a characterization of the distribution of \( (Y_1, Y_2) \) in terms of the joint Laplace transform while leaving the distributions of $X_1$ and $X_2$ unspecified. It is easy to see that if $a_{1,\theta_1}$ and $a_{2,\theta_2}$ are rational in $\theta_1$ and $\theta_2$ respectively, then $\mathcal{L}_{Y_1,Y_2}(\theta_1,\theta_2)$ is rational in \( \theta_1 \) and \( \theta_2 \). Since univariate rational Laplace transform characterizes matrix-exponential distributions (see e.g. Chapter 4 in \cite{bladt2017matrix}) while multivariate rational Laplace transform characterizes the class of MME distributions proposed by \cite{bladt2010multivariate}, we have the following corollary as an immediate consequence.

\begin{corollary}\label{coro:ComSumMME}
If $X_1$ and $X_2$ follow independent matrix-exponential distributions, then the bivariate random vector $(Y_1,Y_2)$ defined in \eqref{eq:Y1Y2def} belongs to the class of $\operatorname{MME}$ distributions introduced in \cite{bladt2010multivariate}.
\end{corollary}

In the following subsection, we establish a related result by showing that if \( (X_1, X_2) \) belongs to the subclass of multivariate phase-type distributions (\( \operatorname{MPH}^* \)) proposed in \cite{kulkarni1989new} then so does \( (Y_1, Y_2) \). The main advantage of \( \operatorname{MPH}^* \) distributions is that, unlike the MME class, they possess a simple probabilistic interpretation.

\subsection{Dependent summands}\label{subsec:dependent-arrivals}

Here we assume that \( (X_1, X_2) \) follows a bivariate version of Kulkarni's \( \operatorname{MPH}^* \) class which we briefly describe next. Such an \( \operatorname{MPH}^* \) distribution is constructed using a single terminating continuous-time Markov jump process $J=\{J(t)\}_{t\ge 0}$ with finite state space $\mathcal{C}$ and a system of state-dependent rewards. Specifically, let \( (\bm{\pi}, \bm{S}, \bm{R}) \) be the parameters, where:

\begin{itemize}
    \item \( \bm{\pi} \) is the initial probability vector;
    \item \( \bm{S} \) is the subintensity matrix governing the continuous-time transitions between transient states; and
    \item \( \bm{R} =\{r_{kj}\}\) is the reward matrix, where \( r_{kj} \) represents the instantaneous non-negative reward associated to the \( j \)-th state in $\mathcal{C}$ and \( k \)-th coordinate for $k\in\{1,2\}$.
\end{itemize}
With these elements, $J$ evolves according to $(\bm{\pi},\bm{S})$ and the coordinates $X_1$ and $X_2$ are defined via
\begin{align*}
X_k:=\int_0^\delta r_{k,J(s)} \,d s,\qquad k\in\{1,2\},
\end{align*}
where $\delta:=\inf\{t\ge 0: J(t)\notin\mathcal{C}\}$ is the termination time of $J$. The random vector \( (X_1, X_2) \) obtained through this construction is said to follow an \( \operatorname{MPH}^*(\bm{\pi}, \bm{S}, \bm{R})\) distribution.

This construction allows for capturing dependencies between \( X_1 \) and \( X_2 \) through the shared underlying Markov process \( J \). Apart from including independence as a special case, the class of \( \operatorname{MPH}^* \) distributions is known to be dense within the class of bivariate distributions with support on the positive quadrant. Unfortunately, the \( \operatorname{MPH}^* \) class is not particularly tractable in the general case as the joint distribution function cannot be expressed in closed form. However, its joint Laplace transform is known to be
\begin{equation*}
\mathbb{E} \left( e^{-\theta_1X_1-\theta_2X_2} \right) = \bm{\pi}\left( \bm{\Delta}_{\theta_1,\theta_2,\bm{R}} - \bm{S} \right)^{-1} (-\bm{S}\bm{1}),
\end{equation*}
for $\theta_1,\theta_2\ge0$, where
\begin{equation*}
\bm{\Delta}_{\theta_1,\theta_2,\bm{R}}=\operatorname{diag}\{\theta_1r_{1j}+\theta_2r_{2j}: j\in\mathcal{C}\}.
\end{equation*}
Our aim is to prove that, under this setup, the bivariate random vector $(Y_1,Y_2)$ defined in (\ref{eq:Y1Y2def}) with each $(X_{1,\ell},X_{2,\ell})$ distributed as the generic bivariate random vector $(X_1,X_2)$ follows an $\operatorname{MPH}^*$ distribution as well.

\begin{proposition}
Suppose that each bivariate random vector in the sequence \(\{(X_{1,\ell}, X_{2,\ell})\}_{\ell \in \mathbb{N}_+}\) follows an \(\operatorname{MPH}^*\) distribution with parameters \( (\bm{\pi}, \bm{S}, \bm{R}) \). Then the pair \((Y_1, Y_2)\) defined in \eqref{eq:Y1Y2def} follows an \(\operatorname{MPH}^*\) distribution with an underlying continuous-time Markov jump process on the state-space \(\big(\mathcal{E} \cup (\mathcal{S} \times \mathcal{S}) \cup (\{\partial_1\} \times \mathcal{S}) \cup (\mathcal{S} \times \{\partial_2\})\big) \times \mathcal{C}\), initial probability vector \((\bm{\alpha}, \bm{0}, \bm{0}, \bm{0}) \otimes \bm{\pi}\), subintensity matrix
\[
\begin{pmatrix}
\bm{I} & \bm{0} & \bm{0} & \bm{0} \\
\bm{0} & \bm{I} & \bm{0} & \bm{0} \\
\bm{0} & \bm{0} & \bm{I} & \bm{0} \\
\bm{0} & \bm{0} & \bm{0} & \bm{I}
\end{pmatrix} \otimes \bm{S} + \begin{pmatrix}
\bm{P} & \bm{U}^* & \bm{0} & \bm{0} \\
\bm{0} & \bm{Q}_1 \otimes \bm{Q}_2 & \bm{q}_1 \otimes \bm{Q}_2 & \bm{Q}_1 \otimes \bm{q}_2 \\
\bm{0} & \bm{0} & \bm{Q}_2 & \bm{0} \\
\bm{0} & \bm{0} & \bm{0} & \bm{Q}_1
\end{pmatrix} \otimes (-\bm{S}\bm{1}\bm{\pi}),
\]
and reward matrix
\[
\begin{pmatrix}
\bm{1}^\intercal \otimes \bm{r}_1 & \bm{1}^\intercal \otimes \bm{r}_1 & \bm{0} & \bm{1}^\intercal \otimes \bm{r}_1 \\
\bm{1}^\intercal \otimes \bm{r}_2 & \bm{1}^\intercal \otimes \bm{r}_2 & \bm{1}^\intercal \otimes \bm{r}_2 & \bm{0}
\end{pmatrix}. %\quad \text{where} \quad \bm{R} = \begin{pmatrix} \bm{r}_1 \\ \bm{r}_2 \end{pmatrix}.
\]
Here, for $k\in\{1,2\}$, the quantity $\bm{r}_k$ represents the row vector $(r_{k1},r_{k2},\ldots,r_{k|\mathcal{C}|})$ such that
\[
\bm{R} = \begin{pmatrix} \bm{r}_1 \\ \bm{r}_2 \end{pmatrix}.
\]
\end{proposition}

\begin{proof}
We follow a probabilistic approach similar to that in Theorem 3.1.28 of \cite{bladt2017matrix}. The key idea is to construct a continuous-time Markov jump process whose accumulated rewards correspond to the random variables \(Y_1\) and \(Y_2\).

Recall \(\gamma\) is the discrete-time Markov chain that leads to the \(\operatorname{MDPH}^*\) representation \eqref{eq:tau_as_Kulkarni}. That is, \(\gamma\) is a Markov chain on the state space \(\mathcal{E} \cup (\mathcal{S} \times \mathcal{S}) \cup (\{\partial_1\} \times \mathcal{S}) \cup (\mathcal{S} \times \{\partial_2\})\). For each \(\ell \in \mathbb{N}_+\), suppose that $J_\ell$ is the process underlying the \(\operatorname{MPH}^*\) pair \((X_{1,\ell}, X_{2,\ell})\). Then, we concatenate independent realizations of $J_\ell$'s to match the number of steps of \(\gamma\) prior to termination. Thus, we construct a process that keeps track of the state of \(\gamma\) while simultaneously evolving within \(\mathcal{C}\). After starting \(\gamma \times J_1\) according to the initial distribution \((\bm{\alpha}, \bm{0}, \bm{0}, \bm{0}) \otimes \bm{\pi}\), the constructed process evolves as follows:
\begin{itemize}
    \item While in \(\mathcal{E} \times \mathcal{C}\) (representing the common-shock phase), \(\gamma\) evolves according to the transition matrix \(\bm{P}\) and \(J_\ell\) evolves independently according to the subintensity matrix \(\bm{S}\). In this phase, both \(Y_1\) and \(Y_2\) accumulate rewards, reflecting that both components are active.

    \item When \(\gamma\) leaves \(\mathcal{E}\), it transitions to \(\mathcal{S} \times \mathcal{S}\) with transition probabilities given by \(\bm{U}^*\), and the constructed process moves into \((\mathcal{S} \times \mathcal{S}) \times \mathcal{C}\) representing the phase after the common shocks. Here, \(\gamma\) evolves according to the Kronecker product \(\bm{Q}_1 \otimes \bm{Q}_2\), while \(J_\ell\) continues to evolve according to \(\bm{S}\). Both \(Y_1\) and \(Y_2\) continue to accumulate rewards.

    \item If one of the components reaches absorption (i.e., transitions to \(\partial_1\) or \(\partial_2\)), the process \(J_\ell\) moves to \(\{\partial_1\} \times \mathcal{S}\) or \(\mathcal{S} \times \{\partial_2\}\), respectively, representing single-coordinate continuation in \((\{\partial_1\} \times \mathcal{S}) \times \mathcal{C}\) or \((\mathcal{S} \times \{\partial_2\}) \times \mathcal{C}\) for the constructed process. In \((\{\partial_1\} \times \mathcal{S}) \times \mathcal{C}\) , only \(Y_2\) continues to accumulate rewards, while the reward accumulation for \(Y_1\) ceases. Conversely, in  \((\mathcal{S} \times \{\partial_2\}) \times \mathcal{C}\), only \(Y_1\) continues to accumulate rewards.

    \item The process terminates when both components have reached absorption. 
\end{itemize}
By constructing the Markov jump process as described and defining the rewards accordingly, we ensure that the accumulated rewards \(Y_1\) and \(Y_2\) correspond to the total occupation times of their respective components before absorption. This setup mirrors the dynamics of the original \(\operatorname{MPH}^*\) distribution for the pair \((Y_1, Y_2)\), thus completing the proof.
\end{proof}

With the random vector $(Y_1,Y_2)$ fully characterized as an \(\operatorname{MPH}^*\) distribution, the joint moments of the pair follow directly from Theorem 8.1.5 in \cite{bladt2017matrix}.

\section{Estimation for bivariate CDPH random variables}\label{sec:estimation}

We shall perform parameter estimation given the data $\{(n_1^{(m)},n_2^{(m)})\}_{m=1}^n$, where each $(n_1^{(m)},n_2^{(m)})$ for $m=1,2,\ldots,n$ is an independent realization from our proposed bivariate CDPH distribution. %For simplicity such data are collected in $\vect{n}_1$ and $\vect{n}_2$, where $\vect{n}_k=(n_k^{(1)},n_k^{(2)},\ldots,n_k^{(n)})$ for $k\in\{1,2\}$.
For later use, we denote $(M_1^{(m)},M_2^{(m)})$ as the joint process underlying the $m$-th CDPH random vector with realization $(n_1^{(m)},n_2^{(m)})$. We first review the fully observed case, and then proceed to provide maximum likelihood estimation through the use of the EM algorithm. %Direct optimization of the likelihood function via the Broyden–Fletcher–Goldfarb–Shanno (BFGS) algorithm is also implemented, but details are standard and thus are omitted.

\subsection{The fully observed case}

Let the parameters be collected in $\mat{\Theta}: = (\bfalp , \mat{P},\bm{U},\mat{Q}_1,\mat{Q}_2)$. The likelihood is given by
\begin{align*}%\label{obs_lik1}
\mathcal{L}^O(\mat{\Theta} \mid \{(n_1^{(m)},n_2^{(m)})\}_{m=1}^n)&
=\prod_{m=1}^nf_{\tau_1,\tau_2}(n_1^{(m)},n_2^{(m)};\mat{\Theta}),
\end{align*}
where we place emphasis on the dependence of the pmf on $\mat{\Theta}$ by writing $f_{\tau_1,\tau_2}(\cdot,\cdot;\mat{\Theta})$ instead of $f_{\tau_1,\tau_2}(\cdot,\cdot)$ (see Proposition \ref{th:mass-taua-taub}).

While direct optimization via gradient methods is possible and relatively quick for very small dimensions, the EM algorithm is more precise and efficient for moderate matrix dimensions $|\mathcal{E}|$ and $|\mathcal{S}|$. Thus, we introduce the following latent variables:
\begin{enumerate}
\item $A_i$: the number of times the processes $\{M_1^{(m)}\}_{m=1}^n$ start in state $i\in\mathcal{E}$.
\item $N^{A}_{ij}$: the total number of jumps of $\{M_1^{(m)}\}_{m=1}^n$ from state $i\in\mathcal{E}$ to $j\in\mathcal{E}$.
\item $N^{T}_{ij}$: the total number of jumps of $\{M_1^{(m)}\}_{m=1}^n$ from state $i\in\mathcal{E}$ to $j\in\mathcal{S}$.
\item $N^{B,k}_{ij}$: the total number of jumps of $\{M_k^{(m)}\}_{m=1}^n$ from state $i\in\mathcal{S}$ to $j\in\mathcal{S}$, where $k\in\{1,2\}$.
\item $N^{B,k}_{i}$: the total number of jumps of $\{M_k^{(m)}\}_{m=1}^n$ from state $i\in\mathcal{S}$ to the absorbing state, where $k\in\{1,2\}$.
\end{enumerate}
Having specified the above sufficient statistics, we are able to reconstruct the paths of each of the Markov jump processes. Recall that $\alpha_i$ is the $i$-th element of the row vector $\bm{\alpha}$ for $i\in \mathcal{E}$ as well as the notation $\bm{P}=\{p_{ij}\}_{i,j \in \mathcal{E}}$, $\bm{Q}_k=\{q_{k,ij}\}_{i,j \in \mathcal{S}}$ and $\bm{U}=\{u_{ij}\}_{i\in \mathcal{E},j\in \mathcal{S}}$. For $k\in\{1,2\}$, we further define $q_{k,i}$ to be the $i$-th element of the column vector $\bm{q}_k$ for $i\in \mathcal{S}$. Then, we can write the complete likelihood as
\begin{align*}
&\mathcal{L}^C( \mat{\Theta}\mid \{M_1^{(m)}\}_{m=1}^n, \{M_2^{(m)}\}_{m=1}^n )\\
&=\left(\prod_{i\in\mathcal{E}} \alpha_i^{A_i}\right) \left(\prod_{i,j\in\mathcal{E}} p_{ij}^{N^{A}_{ij}}\right)
\left(\prod_{i\in\mathcal{E}}\prod_{j\in\mathcal{S}} u_{ij}^{N^{T}_{ij}}\right)
\left(\prod_{k=1}^2\prod_{i,j\in\mathcal{S}} q_{k,ij}^{N^{B,k}_{ij}}\right) 
\left(\prod_{k=1}^2\prod_{i\in\mathcal{S}} q_{k,i}^{N^{B,k}_{i}}\right).
\end{align*}
The above specification is particularly convenient since it belongs to the exponential dispersion family of distributions, and thus has fully explicit maximum likelihood estimators (with a `$\hat{~~}$' placed on top of the corresponding parameter) given by
\begin{align*}
\hat \alpha_{i}&=\frac{A_i}{n},\quad 
\hat p_{ij}=\frac{N_{ij}^A}{
\sum_{\ell\in\mathcal{E}}N_{i\ell}^{A}+\sum_{\ell\in\mathcal{S}}N_{i\ell}^{T}
},\quad
\hat u_{ij}=\frac{N_{ij}^T}{
\sum_{\ell\in\mathcal{E}}N_{i\ell}^{A}+\sum_{\ell\in\mathcal{S}}N_{i\ell}^{T}
},
\\
\hat q_{k,ij}&=\frac{N_{ij}^{B,k}}{
\sum_{\ell\in\mathcal{S}}N_{i\ell}^{B,k} + N_{i}^{B,k}
},\quad
\hat q_{k,i}=\frac{N_{i}^{B,k}}{
\sum_{\ell\in\mathcal{S}}N_{i\ell}^{B,k} + N_{i}^{B,k}
}.
\end{align*}

\subsection{The EM algorithm}\label{subsec:EMalg}

Since the full-trajectory data is not observed, we employ the EM algorithm to find the Maximum Likelihood Estimate iteratively. This implies that at each iteration, the conditional expectations of the sufficient statistics, given the absorption times, are computed, corresponding to the E-step. Subsequently such conditional expectations are replaced into $\mathcal{L}^C$ instead of the statistics themselves, and so upon maximization, we obtain updated parameters $( \mat{\Theta} )$, commonly referred to as the M-step.

The conditional expectations required in the E-step are given as follows: 
\begin{lemma}
Let $\xi:=\{(N^{(m)}_1,N^{(m)}_2)=(n^{(m)}_1,n^{(m)}_2)\}_{m=1}^n$ be the observed data. The conditional expectations required in the E-step are given as follows:
\begin{align*}
    \mathbb{E}[A_i\mid \xi]=\sum_{m=1}^n \frac{\alpha_i f_{\tau_1,\tau_2|M_1,M_2}(n_1^{(m)},n_2^{(m)}|i)}{ f_{\tau_1,\tau_2}(n_1^{(m)},n_2^{(m)}) }\,.
    \end{align*}
\begin{align*}
    \mathbb{E}[N_{ij}^A\mid \xi]=\sum_{m=1}^n \sum_{t\ge1} \frac{(\bfalp \mat{P}^{t-1})_i\,p_{ij} f_{\tau_1,\tau_2|M_1,M_2}(n_1^{(m)}-t,n_2^{(m)}-t|j) }{ f_{\tau_1,\tau_2}(n_1^{(m)},n_2^{(m)}) }\,.
    \end{align*}

\begin{align*}
\mathbb{E}[N_{ij}^{T}\mid \xi]=\sum_{m=1}^n \sum_{t\ge1}
 \frac{
 (\bfalp \mat{P}^{t-1})_i \, u_{ij}  (\mat{Q}_1^{n_1^{(m)}-t-1}\vect{q}_1)_j(\mat{Q}_2^{n_2^{(m)}-t-1}\vect{q}_2)_j
 }{ f_{\tau_1,\tau_2}(n_1^{(m)},n_2^{(m)}) }\,.
\end{align*}
    
\begin{align*}
\mathbb{E}[N_{ij}^{B,1}\mid \xi]=\sum_{m=1}^n \sum_{t\ge2}\sum_{\ell=1}^{t-1}\sum_{r\in\mathcal{S}}
 \frac{
 (\bfalp \mat{P}^{\ell-1}\bm{U})_r (\mat{Q}_1^{t-\ell-1})_{ri} \, q_{1,ij} (\mat{Q}_1^{n_1^{(m)}-t-1}\vect{q}_1)_j (\mat{Q}_2^{n_2^{(m)}-\ell-1}\vect{q}_2)_r
 }{ f_{\tau_1,\tau_2}(n_1^{(m)},n_2^{(m)}) }\,.
\end{align*}

\begin{align*}
\mathbb{E}[N_{ij}^{B,2}\mid \xi]=\sum_{m=1}^n \sum_{t\ge2}\sum_{\ell=1}^{t-1}\sum_{r\in\mathcal{S}}
 \frac{
 (\bfalp \mat{P}^{\ell-1}\bm{U})_r (\mat{Q}_2^{t-\ell-1})_{ri} \, q_{2,ij} (\mat{Q}_2^{n_2^{(m)}-t-1}\vect{q}_2)_j (\mat{Q}_1^{n_1^{(m)}-\ell-1}\vect{q}_1)_r
 }{ f_{\tau_1,\tau_2}(n_1^{(m)},n_2^{(m)}) }\,.
\end{align*}

\begin{align*}
\mathbb{E}[N_{i}^{B,1}\mid \xi]=\sum_{m=1}^n \sum_{t\ge2}\sum_{\ell=1}^{t-1}\sum_{r\in\mathcal{S}}
 \frac{
 (\bfalp \mat{P}^{\ell-1}\bm{U})_r (\mat{Q}_1^{t-\ell-1})_{ri} \, q_{1,i} (\mat{Q}_2^{n_2^{(m)}-\ell-1}\vect{q}_2)_r
 }{ f_{\tau_1,\tau_2}(n_1^{(m)},n_2^{(m)}) }\,.
\end{align*}

\begin{align*}
\mathbb{E}[N_{i}^{B,2}\mid \xi]=\sum_{m=1}^n \sum_{t\ge2}\sum_{\ell=1}^{t-1}\sum_{r\in\mathcal{S}}
 \frac{
 (\bfalp \mat{P}^{\ell-1}\bm{U})_r (\mat{Q}_2^{t-\ell-1})_{ri} \, q_{2,i}  (\mat{Q}_1^{n_1^{(m)}-\ell-1}\vect{q}_1)_r
 }{ f_{\tau_1,\tau_2}(n_1^{(m)},n_2^{(m)}) }\,.
\end{align*}
Here $f_{\tau_1,\tau_2}$ and $f_{\tau_1,\tau_2|M_1,M_2}$ are given by \eqref{eq:phi1} and \eqref{eq:phik1} with the help of \eqref{eq:psi1} and \eqref{eq:psik1}.

\end{lemma}
\begin{proof}
The key in all expressions is to apply the disintegration formula by conditioning on the common shock. All the remaining calculations then follow along similar lines as in the univariate case, since the paths evolve independently thereafter. 
\end{proof}

\section{Simulations and real data analysis}\label{sec:simulation}

In this section, we investigate the flexibility and applicability of our proposed CDPH model on count data. Specifically, we conduct simulations involving bivariate Poisson and Poisson-Lindley distributions, followed by an application to real-world insurance claims data. Our focus is on assessing how well CDPH can capture or mimic these distributions.

\subsection{Bivariate Poisson distribution with common shocks}

We begin by demonstrating that our proposed CDPH model can effectively fit a bivariate Poisson distribution. Although the Poisson distribution itself is not a special case of the DPH distribution, we aim to show that the CDPH model provides a good approximation. 

To define the bivariate Poisson distribution, we let $(N_1^*,N_2^*)$ be a bivariate Poisson distributed pair such that $N_1^* := V_1 + Z$ and $N_2^* := V_2 + Z$, where $V_1$, $V_2$, and $Z$ are independent Poisson random variables with parameters $\lambda_{V_1}$, $\lambda_{V_2}$, and $\lambda_Z$, respectively. Under this construction, the marginal distributions are again Poisson with intensities $\lambda_{N_1^*} := \lambda_{V_1} + \lambda_Z$ and $\lambda_{N_2^*} := \lambda_{V_2} + \lambda_Z$. The joint pmf of $(N_1^*,N_2^*)$ is provided in Equation (2) of \cite{holgate1964bivPoisson}.

To evaluate the effectiveness of our model, we simulate data from $(N_1^*,N_2^*)$ and then fit $(N_1^*+2,N_2^*+2)$ using our CDPH approach\footnote{In this paper, the initial parameters for estimation are randomly generated (independent uniform random variables), suitably transformed to comply with the matrix constraints.}. An interesting aspect of the study is to observe how the proposed model performs as the value of the common shock rate $\lambda_Z$ increases, while keeping $\lambda_{N_1^*}$ and $\lambda_{N_2^*}$ fixed. Specifically, we fix $\lambda_{N_1^*} = 5$ and $\lambda_{N_2^*} = 4$ and consider three different values of $\lambda_Z$, namely 1, 2, and 3. The results are presented in Figures~\ref{fig:pois1}, \ref{fig:pois2}, and \ref{fig:pois3}. For each value of $\lambda_Z$, we explore three different model configurations with dimensions $(|\mathcal{E}|, |\mathcal{S}|) = (2,1)$, $(|\mathcal{E}|, |\mathcal{S}|) = (3,2)$, and $(|\mathcal{E}|, |\mathcal{S}|) = (4,3)$. Since matrix representations are not unique and hence not identifiable, the use of Akaike Information Criterion (AIC) and Bayesian Information Criterion (BIC) is not appropriate in this context. Instead, we track the evolution of the likelihood as a function of the EM iteration step to assess model performance.

The simulations suggest that as the common shock component becomes more pronounced, our model captures it more effectively. Nevertheless, in all cases, the model provides a satisfactory fit, which is quantified through heatmap plots, showing the best fitted pmf among the three considered models (which is the most complex one with dimensions $(|\mathcal{E}|, |\mathcal{S}|) = (4,3)$), as well as the absolute difference between the empirical pmf and the best fitted one.

\begin{figure}[!htbp]
\centering
\includegraphics[width=1\textwidth, trim= 0in 0in 0in 0in,clip]{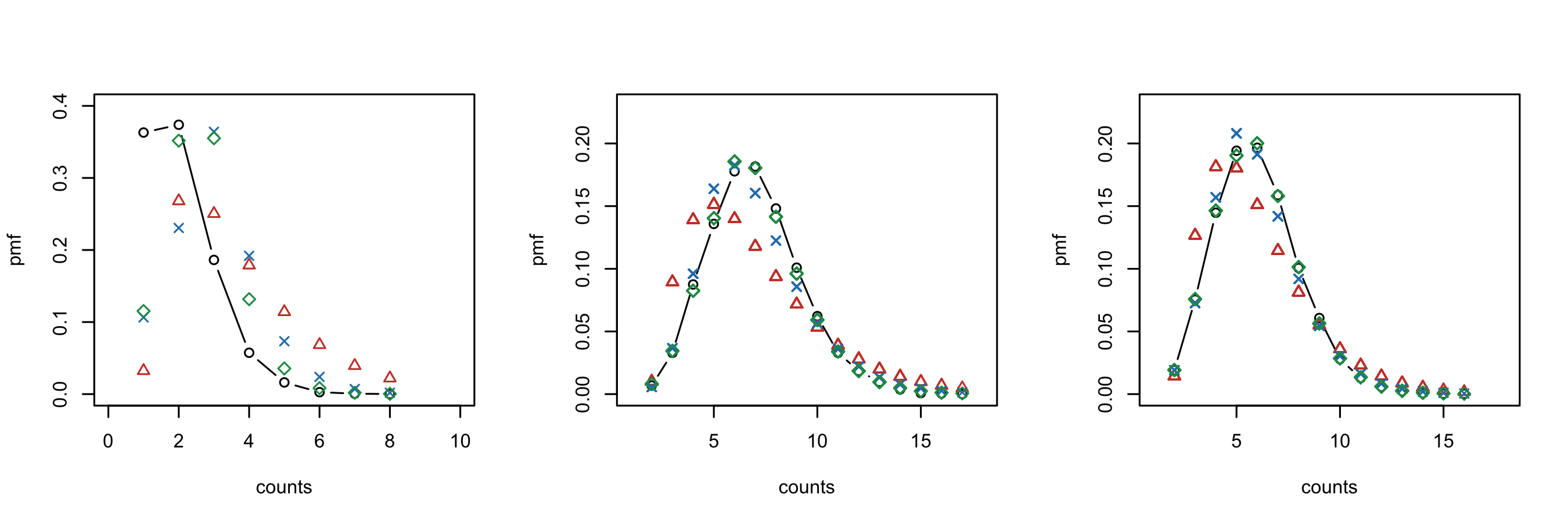}
\includegraphics[width=0.9\textwidth, trim= 0in 0in 0in 0in,clip]{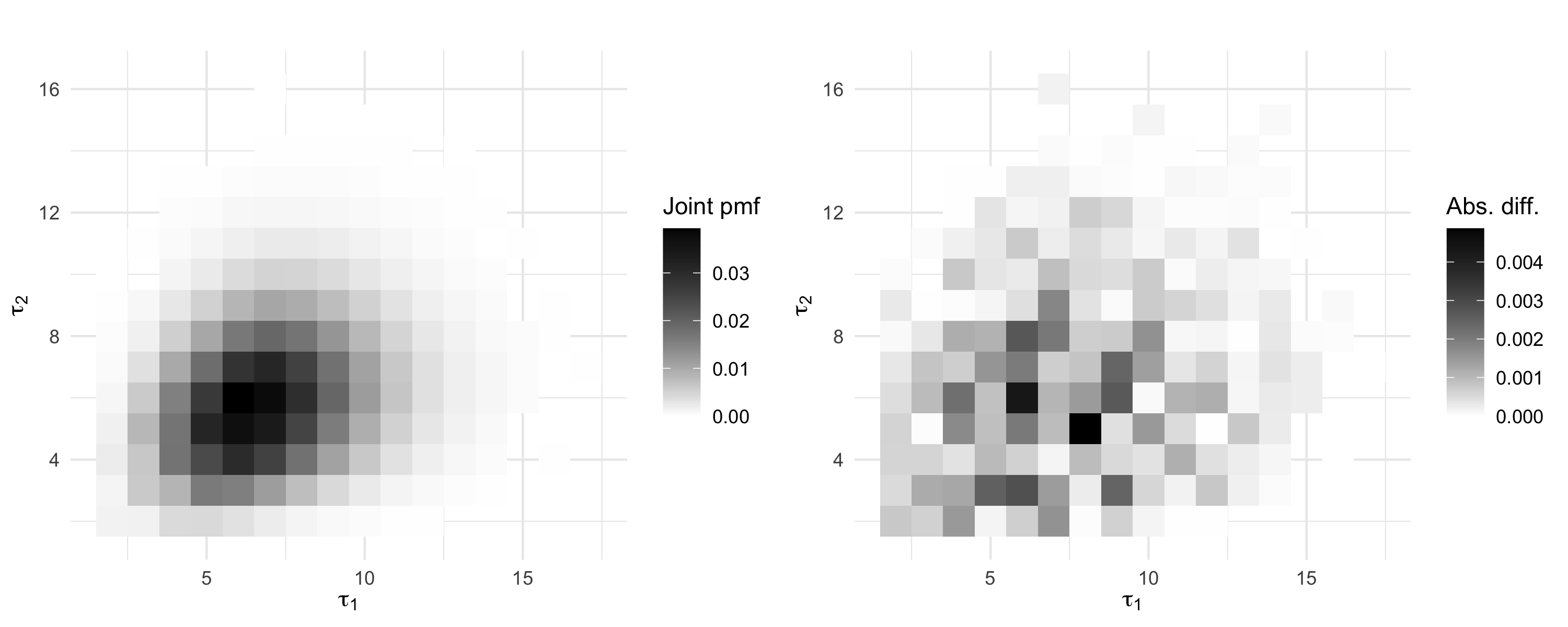}
\includegraphics[width=0.5\textwidth, trim= 0in 0in 0in 0in,clip]{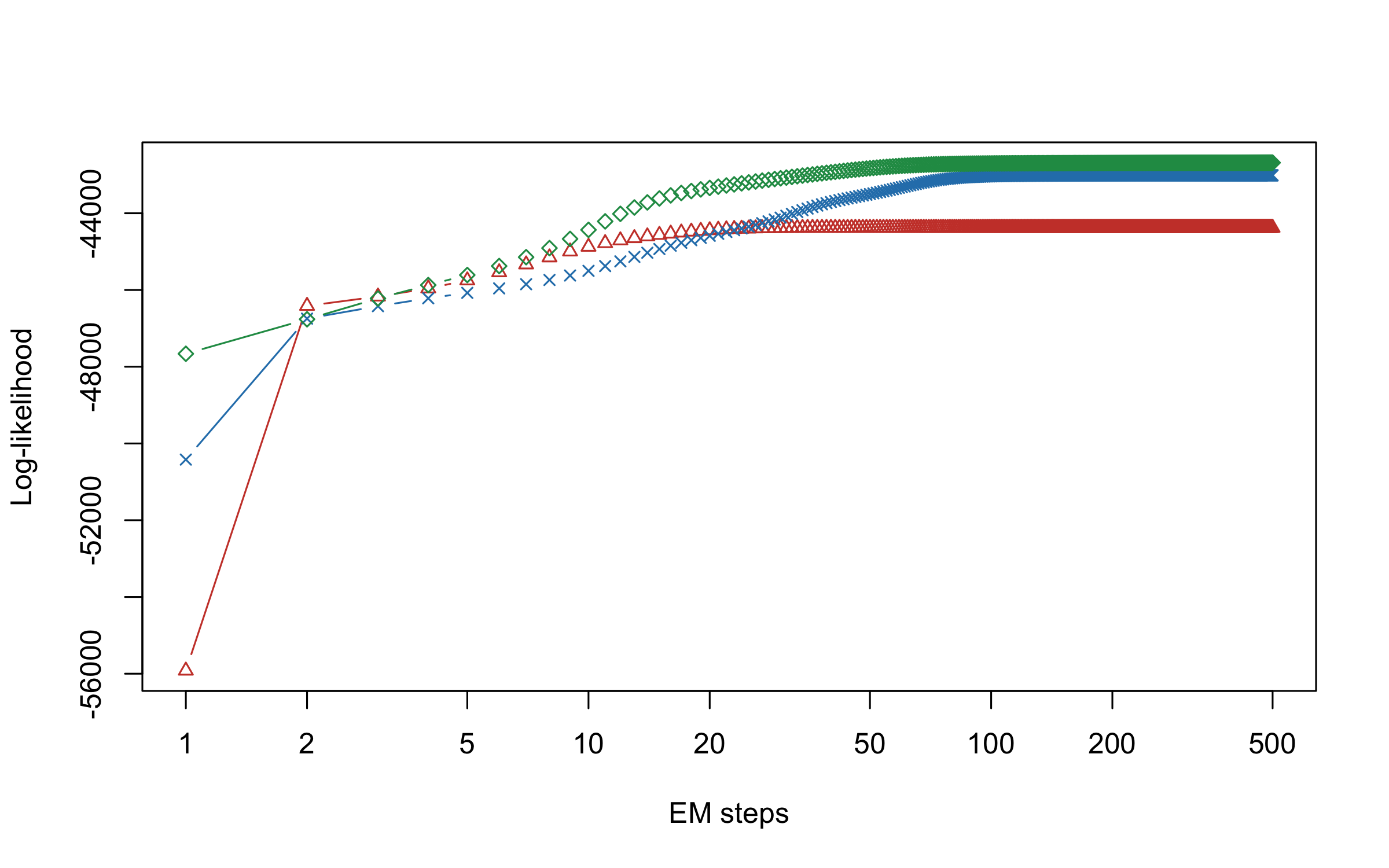}
\caption{Bivariate Poisson distribution with common shock intensity $\lambda_Z=1$. In the top panels we show the empirical pmf (black circles) versus the fitted pmf (by increasing matrix dimension: red triangles, blue crosses, green diamonds), for the distribution of the common shock (left), and the two marginals (center and right). In the middle panel we have two heatmaps, corresponding to the best fitted bivariate pmf (left), and the absolute difference between the best fitted and empirical bivariate pmf (right). The bottom panel shows the evolution of the likelihood for the three models as the number of EM steps increases to $500$.}
\label{fig:pois1}
\end{figure}

\begin{figure}[!htbp]
\centering
\includegraphics[width=1\textwidth, trim= 0in 0in 0in 0in,clip]{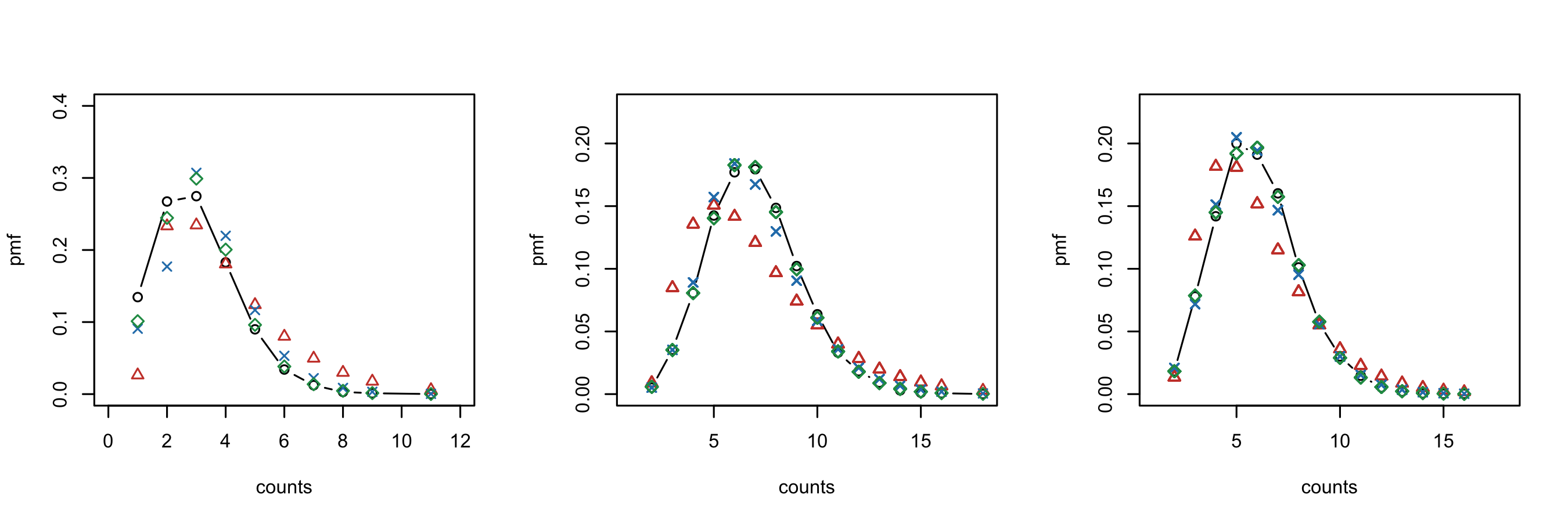}
\includegraphics[width=0.9\textwidth, trim= 0in 0in 0in 0in,clip]{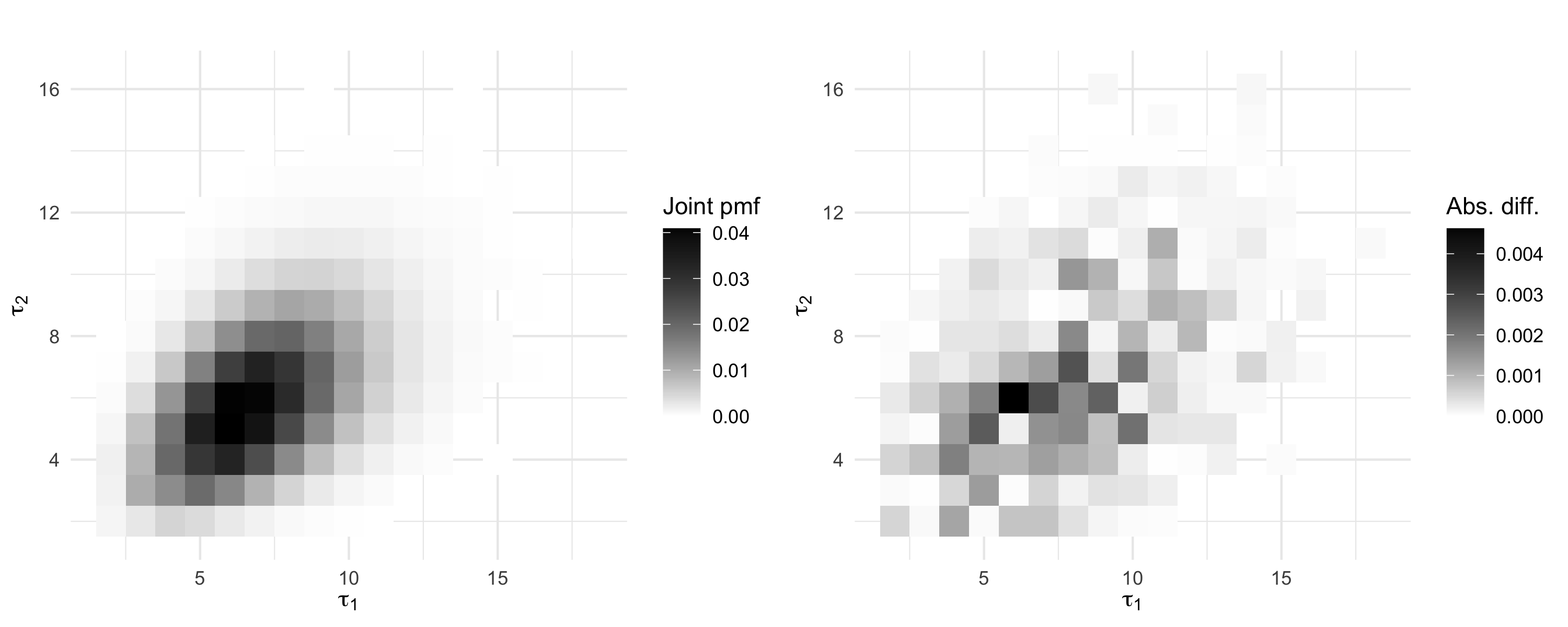}
\includegraphics[width=0.5\textwidth, trim= 0in 0in 0in 0in,clip]{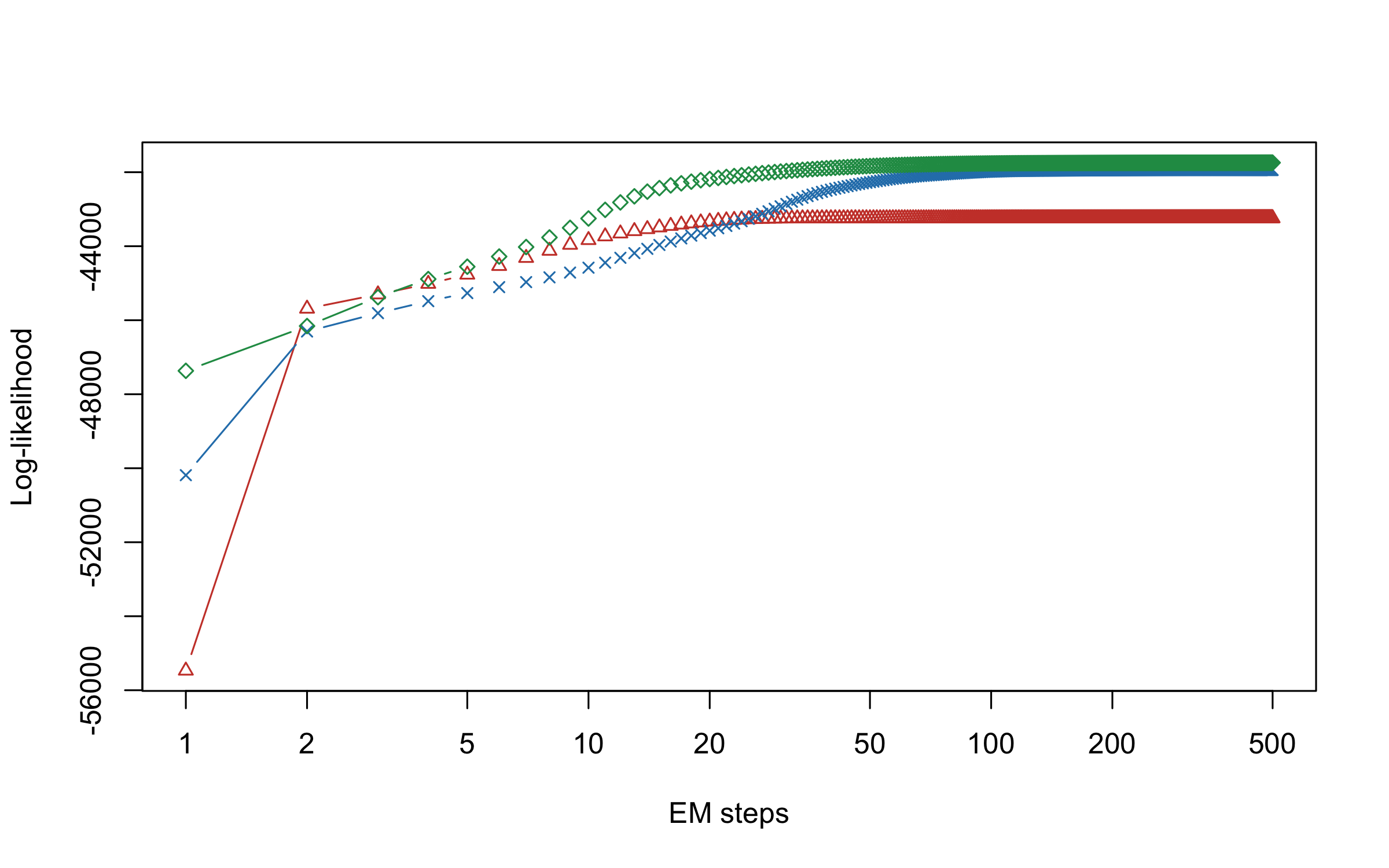}
\caption{Bivariate Poisson distribution with common shock intensity $\lambda_Z=2$. In the top panels we show the empirical pmf (black circles) versus the fitted pmf (by increasing matrix dimension: red triangles, blue crosses, green diamonds), for the distribution of the common shock (left), and the two marginals (center and right). In the middle panel we have two heatmaps, corresponding to the best fitted bivariate pmf (left), and the absolute difference between the best fitted and empirical bivariate pmf (right). The bottom panel shows the evolution of the likelihood for the three models as the number of EM steps increases to $500$.}
\label{fig:pois2}
\end{figure}

\begin{figure}[!htbp]
\centering
\includegraphics[width=1\textwidth, trim= 0in 0in 0in 0in,clip]{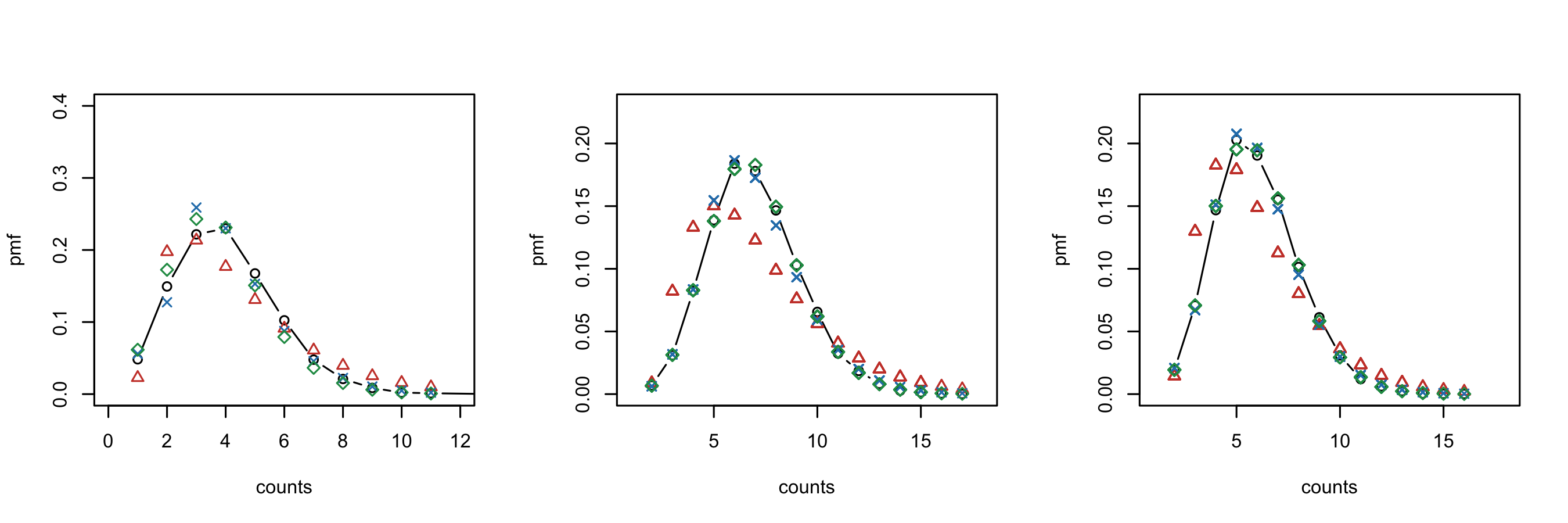}
\includegraphics[width=0.9\textwidth, trim= 0in 0in 0in 0in,clip]{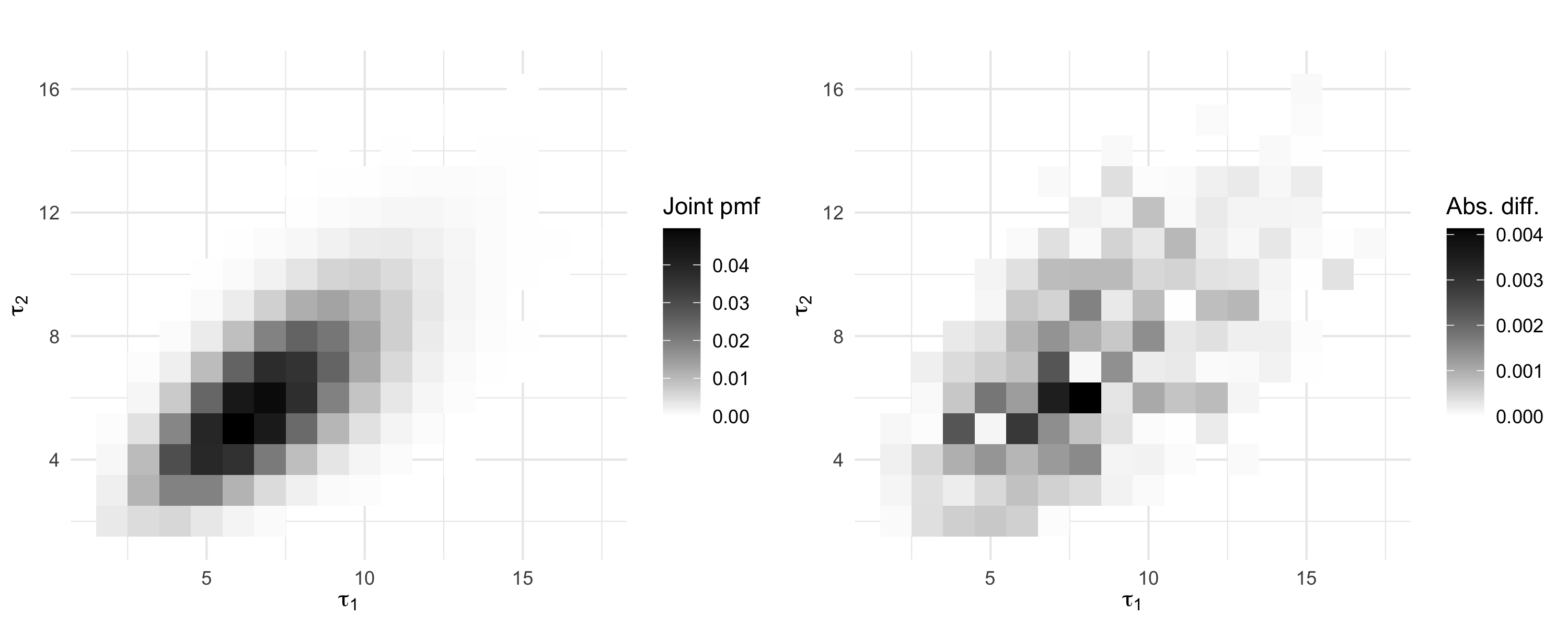}
\includegraphics[width=0.5\textwidth, trim= 0in 0in 0in 0in,clip]{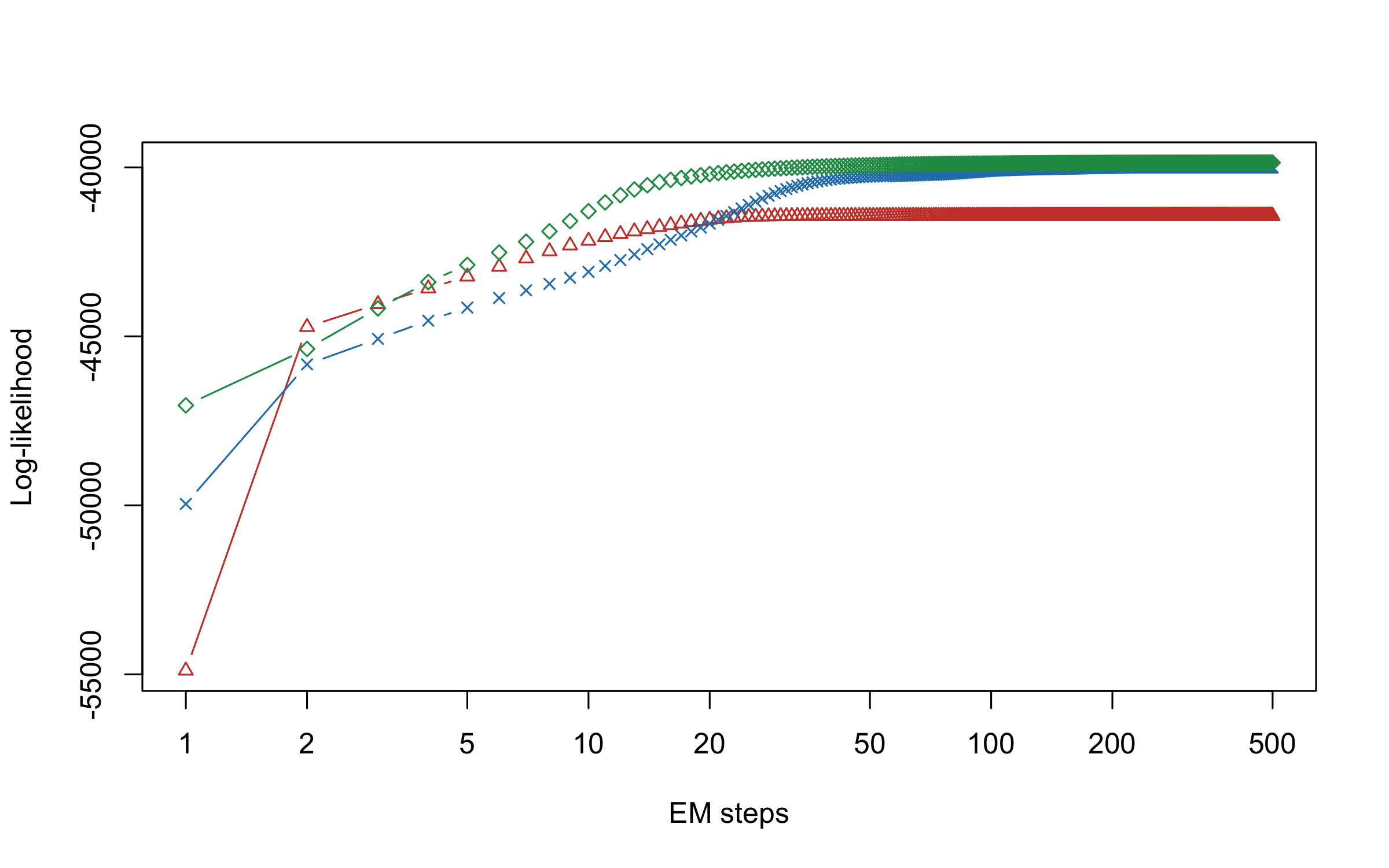}
\caption{Bivariate Poisson distribution with common shock intensity $\lambda_Z=3$. In the top panels we show the empirical pmf (black circles) versus the fitted pmf (by increasing matrix dimension: red triangles, blue crosses, green diamonds), for the distribution of the common shock (left), and the two marginals (center and right). In the middle panel we have two heatmaps, corresponding to the best fitted bivariate pmf (left), and the absolute difference between the best fitted and empirical bivariate pmf (right). The bottom panel shows the evolution of the likelihood for the three models as the number of EM steps increases to $500$.}
\label{fig:pois3}
\end{figure}

\subsection{Bivariate Poisson-Lindley distribution}

In addition to the bivariate Poisson case, we now investigate the ability of our model to capture overdispersion in claim counts, where the variance is not necessarily equal to the mean. The Poisson-Lindley is a mixture distribution that serves as a natural candidate for this setting. Specifically, the bivariate Poisson-Lindley distribution is constructed by considering conditionally independent Poisson random variables, where the Poisson parameter is mixed over a Lindley distribution that depends on a parameter $\theta$ (see Definition 2 in \cite{gomezdeniz2012PoissonLindley}). The joint pmf in the bivariate case is also provided in Equation (17) of \cite{gomezdeniz2012PoissonLindley}.

To examine the performance of our model, we simulate $(N_1^*, N_2^*)$ from this distribution and apply our CDPH model to the transformed data $(N_1^*+2, N_2^*+2)$. Concretely, we simulate a Lindley random variable $L$ with $\theta=2$, and subsequently construct $N_1^*$ and $N_2^*$ as independent Poisson with parameters $2L$ and $3L$, respectively. This is then repeated $10,000$ times. We again consider three model configurations with dimensions $(|\mathcal{E}|, |\mathcal{S}|) = (2,1)$, $(|\mathcal{E}|, |\mathcal{S}|) = (3,2)$, and $(|\mathcal{E}|, |\mathcal{S}|) = (4,3)$. The evolution of the likelihood across EM iterations is tracked, as in the previous study, to assess convergence and fitting accuracy. The results are provided in Figure~\ref{fig:lind}, and they further illustrate the adaptability of our CDPH approach in modeling overdispersed count data. Unlike the bivariate Poisson setting, the Poisson-Lindley mixture does not inherently possess a common shock component. Remarkably, however, our model is still able to effectively capture and mimic dependencies between the two components, demonstrating its practical modeling flexibility.

\begin{figure}[!htbp]
\centering
\includegraphics[width=1\textwidth, trim= 0in 0in 0in 0in,clip]{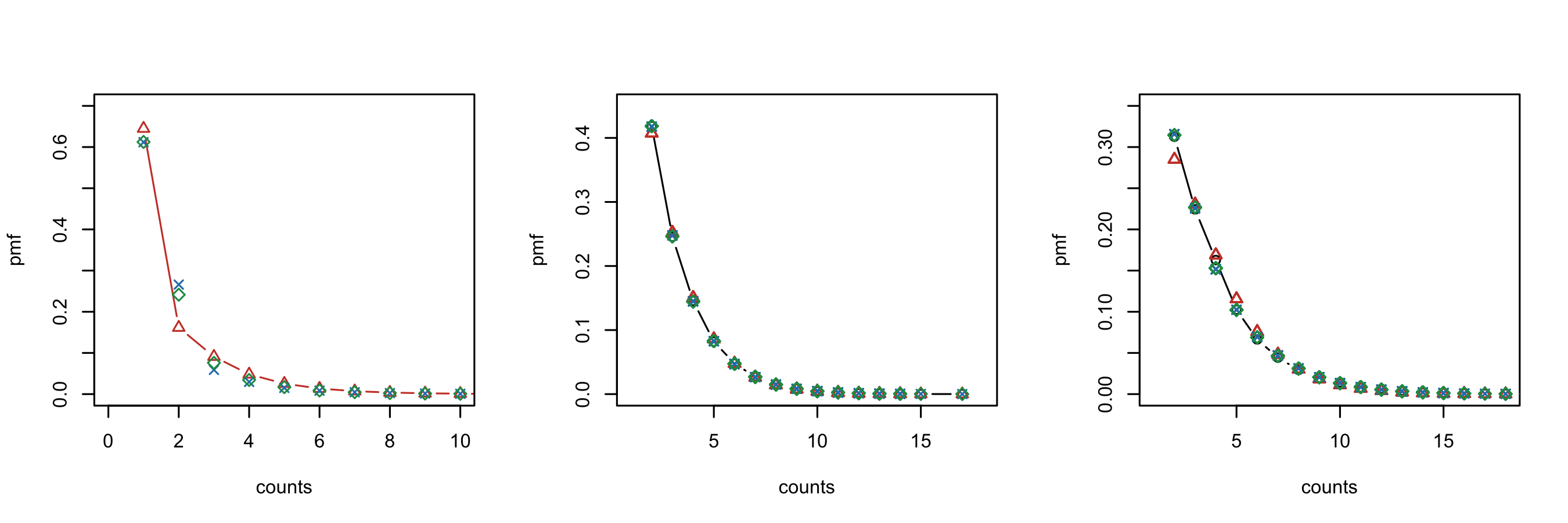}
\includegraphics[width=0.9\textwidth, trim= 0in 0in 0in 0in,clip]{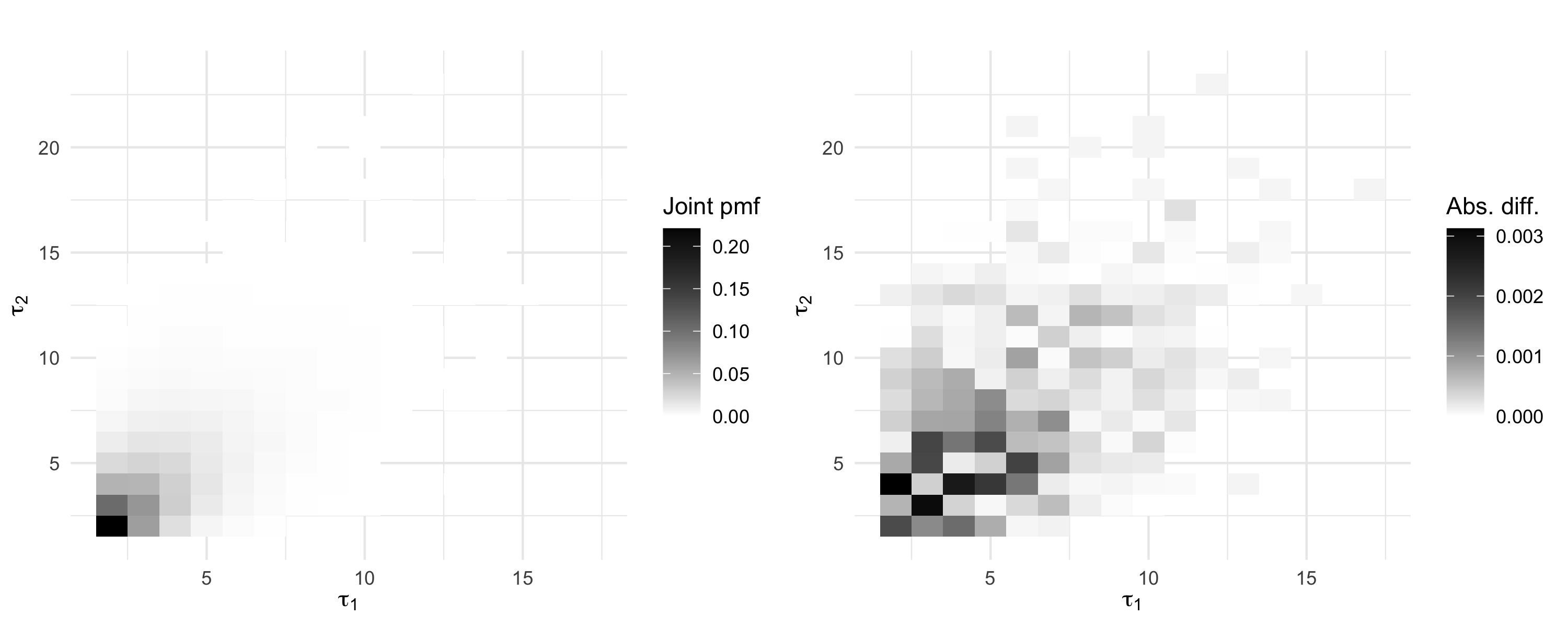}
\includegraphics[width=0.5\textwidth, trim= 0in 0in 0in 0in,clip]{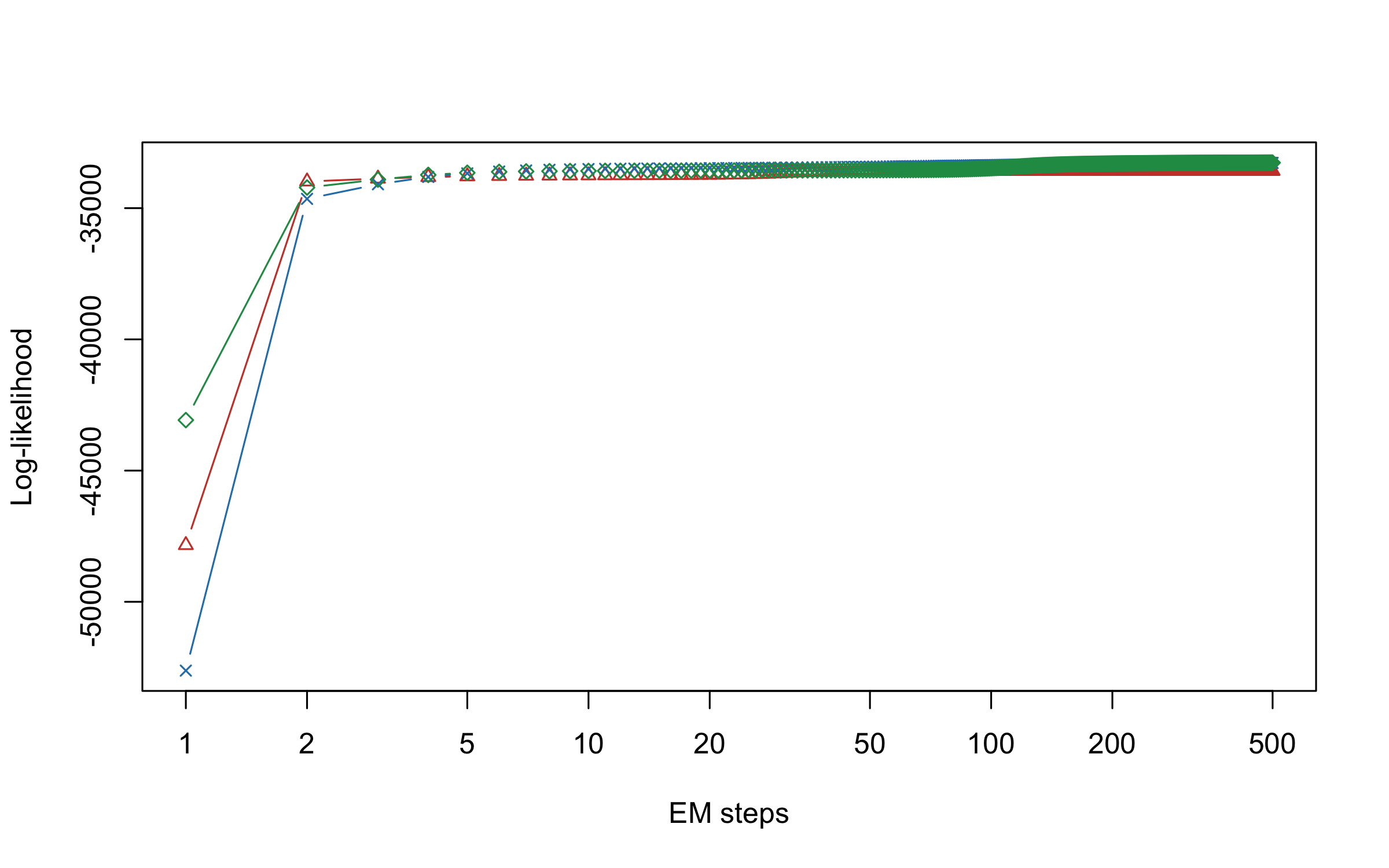}
\caption{Poisson-Lindley distribution with parameter $\theta=2$. In the top panels we show the empirical pmf (black circles) versus the fitted pmf (by increasing matrix dimension: red triangles, blue crosses, green diamonds) for the two marginals (center and right), and the fitted values for a suggested common shock (left). In the middle panel we have two heatmaps, corresponding to the best fitted bivariate pmf (left), and the absolute difference between the best fitted and empirical bivariate pmf (right). The bottom panel shows the evolution of the likelihood for the three models as the number of EM steps increases to $500$.}
\label{fig:lind}
\end{figure}

\subsection{Application to insurance data}

We now apply our model to real-world insurance data, previously analyzed in \cite{vernic2000data}. The dataset can be found in Table 1 of \cite{vernic2000data}, and was originally modeled using the bivariate generalized Poisson distribution (BGPD). It consists of claim frequencies from two related insurance categories, with a total of $708$ claims.

We fit this dataset to our proposed CDPH distribution, applying the same shift transformation as before. The plots in Figure~\ref{fig:real} illustrate the results for the same three different model configurations: $(|\mathcal{E}|, |\mathcal{S}|) = (2,1)$, $(|\mathcal{E}|, |\mathcal{S}|) = (3,2)$, and $(|\mathcal{E}|, |\mathcal{S}|) = (4,3)$. As with the previous cases, we track the likelihood evolution over EM iterations and evaluate the fit using probability plots. The CDPH model in general provides a good fit to the insurance dataset, even in the case of low dimensions. It is interesting to note that a non-negligible common shock component is estimated quite robustly by the three models, which by the previous simulation studies suggests that there could either be a true underlying common shock, or that the bivariate distribution is well approximated by the CDPH class. Despite the structural differences between BGPD and CDPH, our approach is also able to capture dependencies effectively.

\begin{figure}[!htbp]
\centering
\includegraphics[width=1\textwidth, trim= 0in 0in 0in 0in,clip]{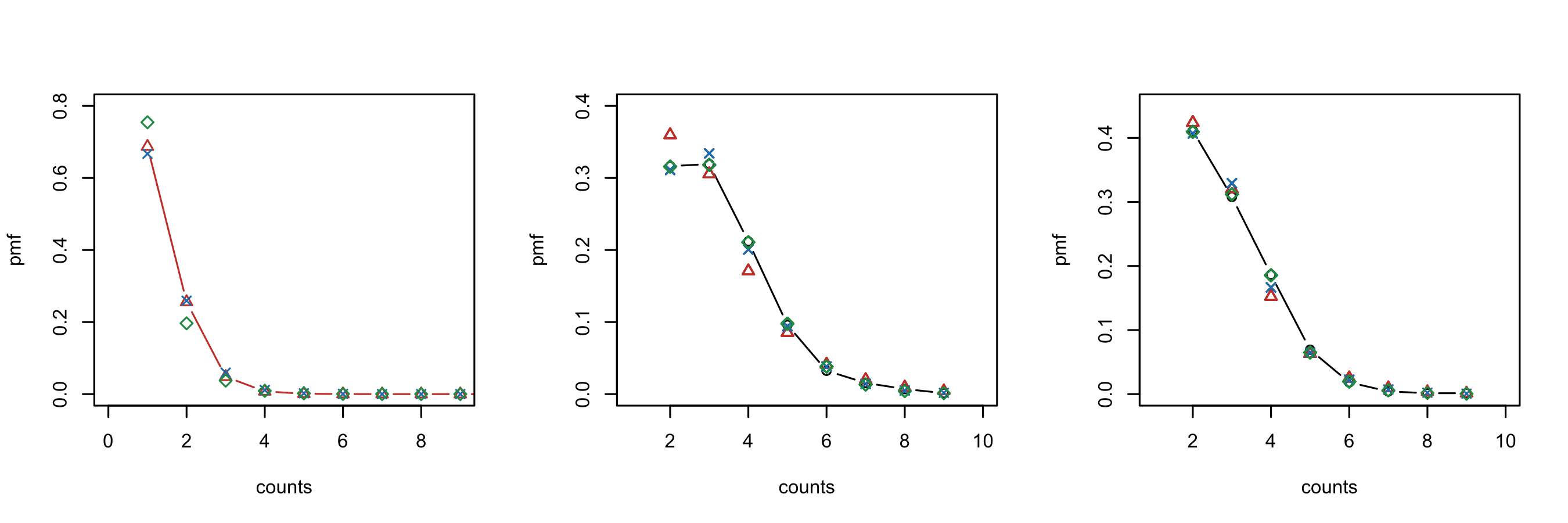}
\includegraphics[width=0.9\textwidth, trim= 0in 0in 0in 0in,clip]{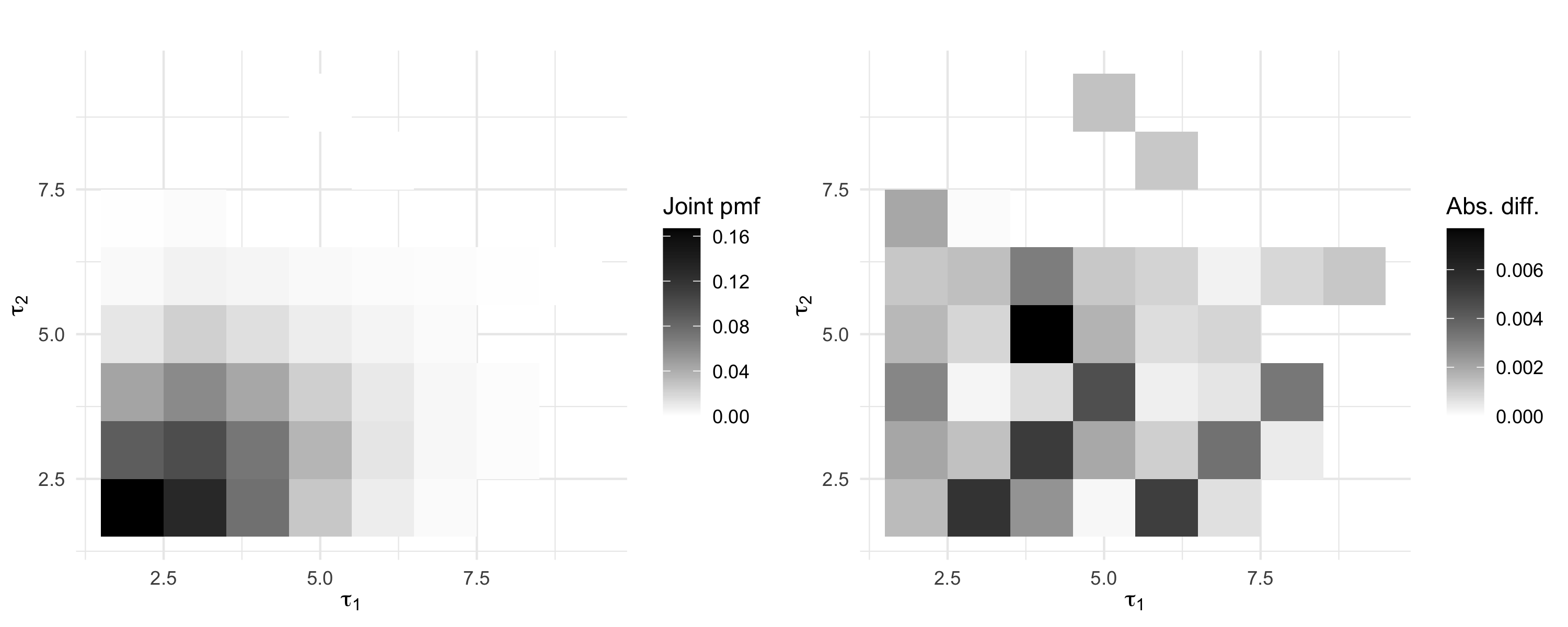}
\includegraphics[width=0.5\textwidth, trim= 0in 0in 0in 0in,clip]{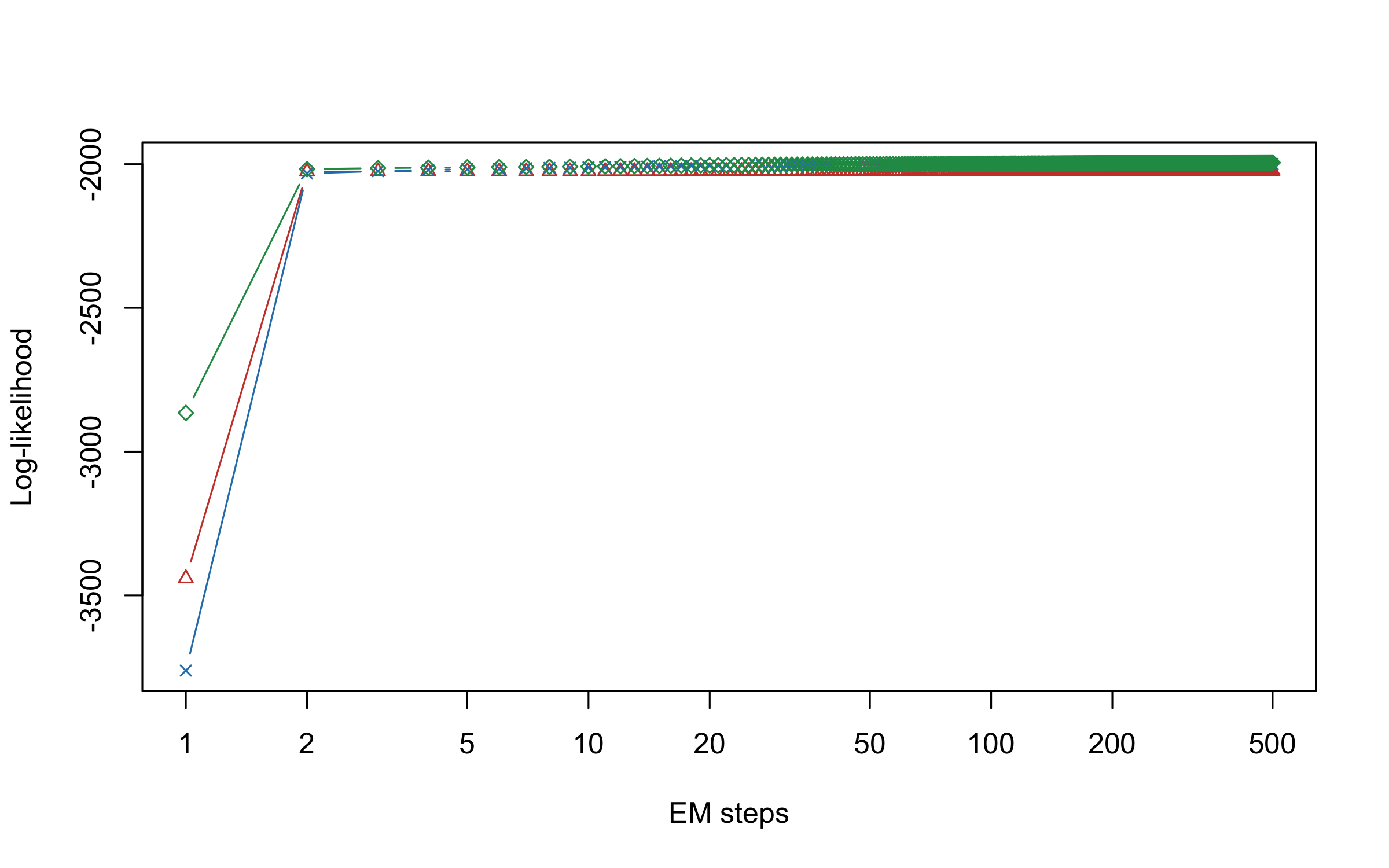}
\caption{Bivariate claim frequency insurance data. In the top panels we show the fitted pmf (by increasing matrix dimension: red triangles, blue crosses, green diamonds), for the distribution of the suggested common shock (left), and the two marginals (center and right). In the middle panel we have two heatmaps, corresponding to the best fitted bivariate pmf (left), and the absolute difference between the best fitted and empirical bivariate pmf (right). The bottom panel shows the evolution of the likelihood for the three models as the number of EM steps increases to $500$.}
\label{fig:real}
\end{figure}

\section{Conclusion}\label{sec:conclusion}

In this paper, we have introduced a novel class of common-shock bivariate discrete phase-type distributions, abbreviated as CDPH, to model dependencies in risk models induced by common shocks. The key idea is that the two random variables $(\tau_1,\tau_2)$ of the CDPH distribution are defined as the termination times of two Markov chains that first evolve jointly (resembling common shocks influenced by shared risk factors) and then proceed independently (following individual-specific dynamics). In addition to deriving explicit and tractable formulas for the joint pmf and the joint pgf, we have also established clear relationships between our CDPH class and other classes of bivariate DPH distributions proposed by \cite{bladt2023robust} and \cite{navarro2019order}. We have proved a number of closure properties of the CDPH class as well. In particular, $\min\{\tau_1,\tau_2\}$, $\max\{\tau_1,\tau_2\}$ and $\tau_1+\tau_2$ are shown to be DPH whereas mixtures and sums of independent and identically distributed copies of $(\tau_1,\tau_2)$ belong to the CDPH class. Further properties of compound sums with CDPH claim counts are also discussed, and an EM algorithm for the CDPH class is developed for parameter estimation and subsequently implemented in our numerical studies.

Our proposed class of bivariate CDPH distributions opens up new research directions for modeling dependencies in multivariate risk modeling, particularly in situations where common shocks play a significant role. Future research could explore extensions to higher dimensions, incorporate covariate information, or consider the tail behavior through the introduction of inhomogeneity functions. We leave these as open questions.

\vskip 1cm
\noindent\textbf{Acknowledgement.} 
MB would like to acknowledge financial support from the Swiss National Science Foundation Project 200021\_191984, as well as from the Carlsberg Foundation, grant CF23-1096. EC and JKW acknowledge the support from the Australian Research Council's Discovery Project DP200100615.

%\newpage
%\appendix
%\section{An appendix}\label{apA}

\newpage

\bibliography{common_shocks.bib}

\newpage

\noindent
Martin Bladt\\
{Department of Mathematical Sciences, University of Copenhagen, Copenhagen, Denmark}\\
{\it Email address:}~{\tt martinbladt@math.ku.dk}
\vskip .5 cm
\noindent
Eric C. K. Cheung\\
{School of Risk and Actuarial Studies, UNSW Business School, University of New South Wales, Sydney, Australia}\\
{\it Email address:}~{\tt eric.cheung@unsw.edu.au}
\vskip .5 cm
\noindent
Oscar Peralta\\
{Department of Actuarial and Insurance Sciences,
Autonomous Technological Institute of M\'exico,
M\'exico City, M\'exico}\\
{\it Email address:}~{\tt oscar.peralta@itam.mx}
\vskip .5 cm
\noindent
Jae-Kyung Woo\\
{School of Risk and Actuarial Studies, UNSW Business School, University of New South Wales, Sydney, Australia}\\
{\it Email address:}~{\tt j.k.woo@unsw.edu.au}

\end{document}